\documentclass[reqno,12pt]{amsart}
\usepackage{mathrsfs}
\usepackage{amscd}
\usepackage{amsmath}
\usepackage{latexsym}
\usepackage{amsfonts}
\usepackage{amssymb}
\usepackage{amsthm}
\usepackage{graphicx}
\usepackage{color}

\def\R{\mathbb{R}}
\def\Z{\mathbb{Z}}
\def\T{\mathbb{T}}

\numberwithin{equation}{section}

\newtheorem{thm}{Theorem}[section]
\newtheorem{lem}{Lemma}[section]

\newtheorem{prop}{Proposition}[section]

\newtheorem{cor}{Corollary}[section]
\newtheorem{remark}{Remark}[section]

\makeatletter
\newcommand{\Extend}[5]{\ext@arrow0099{\arrowfill@#1#2#3}{#4}{#5}}
\makeatother

\begin{document}

\setcounter{page}{1}

\title[Global rough solution for periodic gKdV]{Global well-posedness for periodic generalized\\ Korteweg-de Vries equation}

\author{Jiguang Bao}
\address{School of  Mathematical Sciences, Beijing Normal University, Laboratory of Mathematics and Complex Systems,
Ministry of Education, Beijing 100875, P.R.China}
\thanks{The first author was partially supported by the NSF of China (No.11371060).}
\email{jgbao@bnu.edu.cn}
\author{Yifei Wu}
\email{yifei@bnu.edu.cn}
\thanks{The second author was partially supported by the NSF of China (No. 11101042) and the China Postdoctoral Science
Foundation (No.2012T50068).}


\date{}


\maketitle

\begin{abstract}\noindent
In this paper, we show the global well-posedness for periodic gKdV equations in the space
$H^s(\mathbb{T})$, $s\ge \frac12$ for quartic case,  and $s> \frac59$ for quintic case.  These improve the previous results
of Colliander et al in 2004. In particular, the result is sharp for quintic case. The main approaches are the I-method combining with the resonance decomposition, and a bilinear Strichartz estimate in periodic setting.
\end{abstract}


 \baselineskip=20pt

\section{Introduction}
In this paper, we consider the global well-posedness of the Cauchy problem for the periodic generalized Korteweg-de Vries equations
(gKdV):
\begin{equation}\label{gkdv}
   \left\{ \aligned
    &\partial_t u +  \partial^3_{x} u  =F(u)_x, \quad  (t,x)\in [0,T]\times \T,\\
    &u(0,x)=\phi(x), \quad x\in \T,
   \endaligned
  \right.
 \end{equation}
where $u$ is an unknown real function defined on $[0,T]\times \T$, $\phi$ is a given real-valued function, $F$ is a polynomial of degree $k+1$, and $\T=\R/\Z$
is the circle. For simplicity, we may assume that $F(u)=\mu u^{k+1}$. When $\mu =1$, the equation in (\ref{gkdv}) is
referred to ``defocusing", while when $\mu=-1$ it is referred to ``focusing". For $k=1$ and $k=2$, they are called by the KdV and modified KdV equations, respectively.
These two equations are completely integrable. For $k\ge 3$, they are classified as the generalized KdV equations, which are not completely integrable in general.
In particular, the quartic case $k=3$ and the quintic case $k=4$ are of special interest, which are regarded as the mass-subcritical and mass-critical equations.

The Cauchy problem (\ref{gkdv}) has been widely studied. The periodic KdV and periodic modified KdV equations are well-posedness in $H^s(\T)$ for any $s\geq -\frac{1}{2}$ and $s\ge \frac12$ respectively.  See Kenig, Ponce and Vega \cite{KPV-96} (also \cite{Bourgain, Bourgain-97, KPV-93-kdv} and \cite{KwOh-IMRN} for unconditional well-posedness of modified KdV equation) for local results and Colliander, Keel, Staffilani, Takaoka and Tao (I-team) \cite{CKSTT-03-KDV} for the global results. These above ranges of $s$ are \emph{sharp} in the sense of the uniformly continuous dependence
of the solution on the data, see \cite{KPV-01}. One can also see Kappeler and Topalov \cite{KaTo-CPDE, KaTo-Duke} and the reference therein, for global $C^0$-well-posedness for rougher data. When $F$ is a general polynomial of degree $k+1\geq 4$, it was  shown by I-team \cite{CKSTT-04-gkdvT} the local well-posedness in $H^s(\T)$ for any
$s\geq \frac{1}{2}$. The authors \cite{CKSTT-04-gkdvT} also showed the analytic ill-posedness in $H^s(\T)$ for $s<\frac12$.
So in this sense, the index $\frac12$ is sharp for local well-posedness in Sobolev space $H^s(\T)$.
Moreover, they established the global well-posedness results in $H^s(\T)$
for $s>\frac{13}{14}-\frac{2}{7k}$ (in the defocusing case when $k>4$, which is mass-supercritical).
In particular, they proved that the quartic and quintic gKdV equations are global well-posedness in $H^s(\T)$ whenever
$s>\frac{5}{6}$ and $s>\frac{6}{7}$ respectively. But there exist some gaps to the local threshold $s=\frac{1}{2}$.
In the present paper, we improve the indices and obtain the optimal one for $k=3$ as expected in \cite{CKSTT-04-gkdvT}, while for $k=4$ there is still room to improve to a sharp result.
For the related results in real line case, we just refer to
\cite{CKSTT-03-KDV, Guo, KPV-96,K, Miao-Shao-Wu-Xu:2009:gKdV, Tao-07} for a few of them.
Now our main result can be stated as follows.
\begin{thm} \label{thm:main}
The Cauchy problems of defocusing generalized KdV equations
\begin{equation}\label{gkdv-4}
   \left\{ \aligned
    &\partial_t u +  \partial^3_{x} u  =\partial_x(u^{k+1}), \quad  (t,x)\in [0,T]\times \T,\\
    &u(0,x)=\phi(x), \quad x\in \T
   \endaligned
  \right.
 \end{equation}
are globally well-posed in $H^s(\T)$ with $s\ge \frac12$ for $k=3$, and $s>\frac59$ for $k=4$.
\end{thm}
\begin{remark}
Similar results as Theorem \ref{thm:main} also hold for the focusing equations with the suitable small initial data which guarantees the positivity of the energy.

Moreover, for general nonlinearity, our method here is also available. However,
compared with the local theory, the global result for $k>4$ falls far short of expectations. Even for the quintic case, it
still has gap from the sharp local result.
\end{remark}

The main approach used here is I-method introduced by I-team, see  \cite{CKSTT-02-DNLS, CKSTT-03-KDV,CKSTT-04-gkdvT}  for examples.
Also, we shall use the resonant decomposition argument given in  \cite{Bourgain-04, CKSTT-08}, see also \cite{Miao-Shao-Wu-Xu:2009:gKdV, MWX-DNLS,LWX-NLS,LWX-NLS-T} for more related argument. It is known that the problem
(\ref{gkdv-4}) obeys the conserved Hamiltonian
$$
E(u):=\int_\T \Big(\frac{1}{2}u_x^2+\frac{1}{k+2}u^{k+2}\Big)\,dx.
$$
The scheme of the proof in \cite{CKSTT-04-gkdvT} is to construct the ``almost conservation law" of the ``first" modified
energy $E(Iu)$ by using the I-operator.  Then the global result can be obtained by iteration. Moreover, some suitable
``correction-term" may be added to the first modified energy $E(Iu)$. If this is done, one may define the ``second"
modified energy. Then the better energy increment and global result could be gotten. See \cite{CKSTT-03-KDV}
for a classical application. However, for the gKdV equation, one may note that it is hard to define the second modified
energy in a naive way, via adding a ``correction-term" to $E(Iu)$ directly. The reason is that the
multiplier, introduced to obtain the second modified energy, is singular in the sense that its $L^\infty$-norm is
infinity in a nontrivial set. The same thing happens in the mass-critical gKdV equation in the real line case,
which was considered in \cite{Miao-Shao-Wu-Xu:2009:gKdV}. To get around the difficulty, we employ the resonant
decomposition method. More precisely, we will split the multiplier into ``resonant piece" and ``non-resonant piece", and
then treat them separately. For ``non-resonant piece", we add a ``correction-term"  to define the second modified energy,
while for ``resonant piece",  we prove that it is relatively small. However in periodic case, compared with the real line case in
\cite{Miao-Shao-Wu-Xu:2009:gKdV}, the decomposition should be finer. In fact, we need smaller control in ``resonant piece", because of
the weaker Strichartz estimate in the periodic setting.

Moreover, the result in periodic setting is weaker than the analogous in the real-line setting, and the argument above only is also not enough to obtain the sharp global theory. A main difficulty for the periodic problem is the absence of the bilinear Strichartz estimate. More precisely, in real line setting, one has
\begin{equation*}
\left\|\int\!\!\!\! \int_{\xi=\xi_1+\xi_2}
e^{ix\xi_1+it\xi_1^3}e^{ix\xi_2+it\xi_2^3}
|\xi_1^2-\xi_2^2|^{\frac{1}{2}}\widehat{\phi_1}(\xi_1)\widehat{\phi_2}(\xi_2)\,d
\xi_1 d \xi\right\|_{L^2_{xt}}\lesssim
\|\phi_1\|_{L^2}\|\phi_2\|_{L^2}.
\end{equation*}
But it doesn't work in periodic setting. So a novelty in this paper is a variant bilinear Strichartz estimate in the periodic case.

In
\cite{CKSTT-03-KDV}, the authors established a type of Strichartz estimate on the long period linear-flow,
which is available to the rescaled problem. The kind of the bilinear Strichartz estimate for the periodic nonlinear Schr\"odinger
equation was
established simultaneously in \cite{BuGeTz, DPST-DCDS-07},
and see also \cite{LWX-NLS-T} for an improvement version to the one in \cite{DPST-DCDS-07}. Inspired from these papers, we
obtain the following bilinear Strichartz estimate on Airy equation here,
\begin{equation*}
\left\|\int\!\!\!\! \int_{\xi=\xi_1+\xi_2}
e^{ix\xi_1+it\xi_1^3}e^{ix\xi_2+it\xi_2^3}
C(\xi_1,\xi_2,\lambda)\widehat{\phi_1}(\xi_1)\widehat{\phi_2}(\xi_2)\,
\big(d\xi_1\big)_\lambda \big(d\xi\big)_\lambda\right\|_{L^2_{xt}}\lesssim
\|\phi_1\|_{L^2}\|\phi_2\|_{L^2},
\end{equation*}
where $(d\xi)_\lambda$ denotes the normalized counting measure defined in Section 2, and the function
$$
C(\xi_1,\xi_2,\lambda)\rightarrow |\xi_1^2-\xi_2^2|^{\frac{1}{2}}, \quad
\mbox{as }\lambda \rightarrow \infty.
$$
It matches the bilinear Strichartz estimate in real line case when the period (or the scaling parameter) tends to infinity.
However, for the rescaled problems under study, the efficacy of the bilinear Strichartz estimate in the periodic case is exactly weaker than the one in the real line case.
See Remark \ref{rem:Bi-Str} below.
This is different from the one on the periodic Schr\"odinger equation (see \cite{DPST-DCDS-07, LWX-NLS-T}). 
%


\subsection{Outline of the proof}

\subsubsection{Working space}

First, we use the gauge transformation introduced in \cite{Bourgain, Staffilani}. Let
$$
\mathcal G u(t,x)=u\Big(t,x+\int_0^t\!\!\!\int_\T u^k\ dxds\Big).
$$
Then we denote functional space $X^s$ as our working space, which is equipped by the norm,
\begin{equation}\label{workingspace}
\|u\|_{X^s}:=\|\mathcal G u\|_{Y^s},
\end{equation}
where $Y^s$ is the standard (but slightly modified) Bourgain space defined in Section 2. Moreover, we denote $X^s(I)$ to be its restricted space on time interval $I$.

In \cite{CKSTT-04-gkdvT}, I-team employed this gauge transformation to avoid
a nontrivial resonance in the original equation. Under this transform, the function $\mathcal G u$ satisfies the equation
\begin{equation}\label{eqs:v}
\partial_t v+\partial_{x}^3v=\mathbb{P}\big[\mathbb{P}(v^k)v_x\big],
\end{equation}
where $\mathbb P$ denotes the orthogonal projection onto mean zero functions,
$$
\mathbb{P}f=f-\int_\T f\,dx,
$$
that is, $\widehat{\mathbb{P}f}(0)=0$. Then the authors considered $\mathcal G u$ instead to prove the sharp local well-posedness via multilinear estimates.

However,  to study the global theory, we can not employ the forms of \eqref{eqs:v} because it breaks the symmetries, which gives the bad form of the modified energies and thus against finer multiplier estimates, see Step 3 below.
So we still consider the original equation, but use the gauged norm \eqref{workingspace} which is dependent upon the local theory. 
This causes difficulties in the multiplier estimates. Fortunately, 
this difficulty can be overcome by using some good properties of the gauge transformation.
For this reason, one shall be careful in the usage of the Bourgain norm.

\subsubsection{$I$-operator}
Let $N\gg 1$ be fixed, and the Fourier multiplier operator
$I_{N,s}$ be defined as
\begin{equation}
\widehat{I_{N,s}f}(\xi)=m_{N,s}(\xi)\hat{f}(\xi).\label{I}
\end{equation}
Here the multiplier $m_{N,s}(\xi)$ is a smooth, monotone function
satisfying $0<m_{N,s}(\xi)\leq 1$ and
\begin{equation}
m_{N,s}(\xi)=\biggl\{
\begin{array}{ll}
1,&| \xi|\leq N,\\
N^{1-s}| \xi|^{s-1},&| \xi|>2N.\label{m}
\end{array}
\end{equation}
Usually, we denote $I_{N,s}$ and $m_{N,s}$ as $I$ and $m$
respectively for short if there is no confusion. Then
\begin{equation}
    \|f\|_{H^s}\lesssim \|I_{N,s}f\|_{H^1}
\lesssim
    N^{1-s}\|f\|_{H^s}.\label{II}
\end{equation}

\subsubsection{Sketch the proofs}

Now we sketch the proof of Theorem \ref{thm:main} in the following steps.

\texttt{Step 1: Rescaling.}

We rescale the problem by writing
$$
u_\lambda(t,x)=\lambda^{-\frac{2}{k}}u(t/\lambda^3,x/\lambda);\quad
\phi_{\lambda}(x)=\lambda^{-\frac{2}{k}}\phi(x/\lambda),
$$
then $u_\lambda$ satisfies that
\begin{equation}\label{gkdvR}
   \left\{ \aligned
    &\partial_t u_\lambda +  \partial^3_{x} u_\lambda =(u_\lambda^{k+1})_x, \quad (t,x)\in [0,\lambda^3T]\times [0,\lambda],\\
    &u_\lambda(0,x)=\phi_\lambda(x),\quad x\in [0,\lambda].
   \endaligned
  \right.
 \end{equation}
Moreover, the solution of (\ref{gkdv}) $u$ exists on $[0,T]$ if and
only if $u_\lambda$ exists on $[0,\lambda^3T]$. On the other hand,
we get that for any $q\geq 1$ and
$s\geq 0$,
\begin{align}\label{9.51}
\|\phi_\lambda\|_{L^q_x}
=\lambda^{\frac{1}{q}-\frac{2}{k}}\|\phi\|_{L^q_x}; \quad
\|\phi_{\lambda}\|_{\dot{H}^s} = \lambda^{\frac{1}{2}-\frac{2}{k}-s}
 \|\phi\|_{\dot{H}^s}.
\end{align}
Hence, by (\ref{II}) and $m(\xi)\leq 1$,
\begin{equation*}\aligned
E(I\phi_\lambda)=&\; \frac{1}{2}\|\partial_x
I\phi_\lambda\|_{L^2}^2+\frac{1}{k+2}\|I\phi_\lambda\|_{L^{k+2}}^{k+2}\\
\lesssim&\;
N^{2-2s}\|\phi_{\lambda}\|_{\dot{H}^s}^2+\|\phi_\lambda\|_{L^{k+2}}^{k+2}\\
\lesssim&\; N^{2-2s}/\lambda^{\frac{4}{k}+2s-1}\cdot
\|\phi\|_{H^s}+\lambda^{-1-\frac{4}{k}}\|\phi\|_{L^{k+2}}^{k+2}.
\endaligned\end{equation*}
To normalize the rescaled initial data, we choose
\begin{equation}\label{lambda}
\lambda\sim N^{\frac{1-s}{\frac{2}{k}+s-\frac{1}{2}}}.
\end{equation}
Then,
\begin{equation}\label{B of Iu}
 \|I\phi_\lambda\|_{H^1},\,\,E(I\phi_\lambda) \lesssim 1.
\end{equation}

\texttt{Step 2: Local theory  for rescaled solutions}. We need the following local theory,
\begin{lem}[\cite{CKSTT-04-gkdvT}] \label{lem:modified-local}
Let $s\geq \frac{1}{2}$ and $\phi$ satisfy $\|I\phi_\lambda\|_{H^1}\lesssim 1$, then Cauchy problem (\ref{gkdvR}) is
locally well-posed on the interval $[0,\delta]$
with the lifetime
\begin{equation}
 \delta\sim \lambda^{-\epsilon}\label{delta}
\end{equation}
for some small $\epsilon>0$. Furthermore, the solution satisfies the estimate
\begin{equation}
    \|Iu_\lambda\|_{X^1([0,\delta])}
    \lesssim
    \|I\phi_\lambda\|_{H^1}.\label{LSE}
\end{equation}
\end{lem}

\texttt{Step 3: Definition of modified energies}. It will be convenient to define
$$
f_\lambda(t):= e^{t\partial_x^3}u_\lambda(t),
$$
then one may find that
\begin{equation*}
\partial_t f_\lambda=e^{t\partial_x^3}\partial_x\big(u_\lambda^{k+1}\big).
\end{equation*}
Therefore, we have
\begin{equation}\label{eqs:1719.47}
\partial_t \widehat{f_\lambda}(\xi)=i\xi\int_{\xi_1+\cdots+\xi_{k+1}=\xi}e^{i(-\xi^3+\xi_1^3+\cdots+\xi_{k+1}^3)t}\widehat{f_\lambda}(t,\xi_1)\cdots\widehat{f_\lambda}(t,\xi_{k+1})
(d\xi_1)_\lambda\cdots (d\xi_{k})_\lambda.
\end{equation}

We denote $m_j=m( \xi_j)$, $\alpha_{k+2}= \xi_1^3+\cdots+ \xi_{k+2}^3$,  $\Gamma_{n}$ to be the hyperplane
\begin{equation}\label{Gamma_n}
\Gamma_n=\left\{( \xi_1,\cdots,  \xi_n)\in \left(\frac\Z\lambda\right)^n: \xi_1+\cdots+ \xi_n=0\right\}.
\end{equation}
From Plancherel's identity (see \eqref{Plancherel} below), it follows that
\begin{align*}
E(Iu_\lambda)=&\frac{1}{2}\int_0^\lambda |\partial_xIu_\lambda(t,x)|^2\,dx+\frac{1}{k+2}\int_0^\lambda |Iu_\lambda(t,x)|^{k+2}\,dx\\
=&\frac12\int_{\Gamma_2} m_1^2 \xi_1^2\> \widehat{f_\lambda}(t,\xi_1)\widehat{f_\lambda}(t,\xi_2)(d\xi_1)_\lambda\\
&+\frac{1}{k+2}\int_{\Gamma_{k+2}}m_1\cdots m_{k+2}\,e^{i\alpha_{k+2}t}
\widehat{f_\lambda}(t,\xi_1)\cdots\widehat{f_\lambda}(t,\xi_{k+2})
(d\xi_1)_\lambda\cdots (d\xi_{k+1})_\lambda.
\end{align*}
By \eqref{eqs:1719.47}, the symmetries of the variables $\xi_j$ in the integration and a direct computation,  we have
\begin{align}
&\dfrac{d}{dt}E(Iu_\lambda(t))\notag\\
=&\int_{\Gamma_{k+2}} e^{i\alpha_{k+2}t}\big(M_{k+2}+i\sigma_{k+2}\>\alpha_{k+2}\big)
\widehat{f_\lambda}(t,\xi_1)\cdots\widehat{f_\lambda}(t,\xi_{k+2})
(d\xi_1)_\lambda\cdots (d\xi_{k+1})_\lambda\label{dE1}\\
&+\int_{\Gamma_{2k+2}} e^{i\alpha_{2k+2}t}M_{2k+2}\>
\widehat{f_\lambda}(t,\xi_1)\cdots\widehat{f_\lambda}(t,\xi_{2k+2})
(d\xi_1)_\lambda\cdots (d\xi_{2k+1})_\lambda,\label{dE2}
\end{align}
where
\begin{equation*}\aligned
M_{k+2}( \xi_1,\cdots, \xi_{k+2}) :=&i\big(m_1^2\xi_1^3+\cdots+m_{k+2}^2\xi_{k+2}^3\big);
\quad \sigma_{k+2}:=\frac1{k+2} m_1\cdots m_{k+2};\\
M_{2k+2}( \xi_1,\cdots, \xi_{2k+2}) :=&\; i(k+2)[\sigma_{k+2}( \xi_1,\cdots, \xi_{k+1},
\xi_{k+2}+\cdots+ \xi_{2k+2}) ( \xi_{k+2}+\cdots+ \xi_{2k+2})]_{sym},
\endaligned\end{equation*}
and $[m]_{sym}$ denotes the symmetrization of a multiplier $m$ (see
\cite{CKSTT-03-KDV}).  So far, this precess is rather standard, and it is the same as what in real line case, see \cite{F, Miao-Shao-Wu-Xu:2009:gKdV}, etc..

Now we focus our attention on the term \eqref{dE1}, and consider the quantity
\begin{equation}\label{M-phi}
\frac{M_{k+2}}{\alpha_{k+2}}.
\end{equation}
If it makes sense, then one may use the identity
\begin{equation}\label{identity}
e^{i\alpha_{k+2}s}=\frac{1}{i\alpha_{k+2}} \partial_s\big(e^{i\alpha_{k+2}s}\big),
\end{equation}
and take the derivative in $s$. This way gives the definition of the second modified energy, and may improve the tiny increment estimate of $E(Iu)$.

One may find \eqref{M-phi} is bounded when $k=1,2$.
But unfortunately, \eqref{M-phi} is singular and thus does not make sense in general when $k\ge 3$.  So it fails to define the second modified energy in this way. Here our argument is the resonance decomposition developed in \cite{Miao-Shao-Wu-Xu:2009:gKdV}.

To do this, we first make a convenient reduction. Denote $\xi_1^*,\cdots,\xi_{k+2}^*,\cdots, \xi_{2k+2}^*$ to be the rearrangement of $\xi_1,\cdots,\xi_{k+2}, \cdots,\xi_{2k+2}$, with
$
|\xi_1^*|\ge \cdots \ge  |\xi_{k+2}^*|\ge \cdots\ge |\xi_{2k+2}^*|.
$
\begin{remark}[A convenient reduction]\label{rem:reduction}
If $|\xi_1^*|\ll N$, then $M_{k+2}, M_{2k+2}=0$ which gives the conservation of $E(Iu_\lambda)$. Hence one may restrict $|\xi_1^*|\gtrsim N$ in the support of $\Gamma_{k+2}$ and $\Gamma_{2k+2}$.
\end{remark}
Now we define the ``non-resonance'' set  using the spirit in \cite{Miao-Shao-Wu-Xu:2009:gKdV}, let
$$
\Omega=\Omega_1\cup \Omega_2\cup \Omega_3\cup \Omega_4,
$$
where
\begin{align*}
\Omega_1=&\big\{(\xi_1,\cdots,\xi_{k+2})\in \Gamma_{k+2}: |\xi_3^*|\gg |\xi_4^*|\big\};\\
\Omega_2=&\big\{(\xi_1,\cdots,\xi_{k+2})\in \Gamma_{k+2}: |\xi_1^*|\sim |\xi_2^*|\gtrsim N\gg |\xi_3^*|\sim |\xi_4^*|,
\big|{\xi_1^*}^3+{\xi_2^*}^3\big|\gg \big|{\xi_3^*}^3+\cdots+{\xi_{k+2}^*}^3\big|\big\};\\
\Omega_3=&\big\{( \xi_1,\cdots,  \xi_{k+2})\in \Gamma_{k+2}:
|\xi_1^{*}|\gg |\xi_3^{*}|,|\xi_1^{*}+\xi_2^{*}||\xi_1^*|\gg |\xi_3^*|^2\big\};\\
\Omega_4=&\big\{(\xi_1,\cdots,\xi_{k+2})\in \Gamma_{k+2}: |\xi_4^*|\gg|\xi_5^*|,\big|\xi_1^*+\xi_2^*\big|\big|\xi_1^*+\xi_3^*\big|\big|\xi_1^*+\xi_4^*\big|\gg |\xi_5^*||\xi_1^*|^2\\
&\qquad \big|m(\xi_1^*)^2{\xi_1^*}^3+\cdots+m(\xi_4^*)^2{\xi_4^*}^3\big|\gg \big|m(\xi_5^*)^2{\xi_5^*}^3+m(\xi_6^*)^2{\xi_6^*}^3\big|\big\},
\end{align*}
and $\xi_6^*=0$ if $k=3$.

Compared with the ``non-resonance'' sets defined in \cite{Miao-Shao-Wu-Xu:2009:gKdV}, we add the set $\Omega_3$ and slightly change the definition on $\Omega_4$. They are employed to overcome the trouble
from the weak Strichartz estimates in the periodic setting.

Firstly, \eqref{M-phi} is bounded in ``non-resonance'' set, that is,
\begin{lem}\label{non-res}
\begin{align}\label{eqs:non-res}
|M_{k+2}|\lesssim |\alpha_{k+2}|,\quad \mbox{ in } \Omega.
\end{align}
\end{lem}
Secondly, we have
\begin{lem}\label{es:Mk+2}
In $\Gamma_{k+2}\backslash\Omega$,
\begin{itemize}
\item[(1).] It holds that
\begin{align}\label{lem:Mk+2-1}
|M_{k+2}|\lesssim m(\xi_1^*)^2|\xi_1^*||\xi_3^*|^2.
\end{align}
\item[(2).] If
$|\xi_1^*|\sim |\xi_2^*|\gtrsim N\gg |\xi_3^*|\sim |\xi_4^*|$, then
\begin{align}\label{lem:Mk+2-2}
|M_{k+2}|\lesssim |\xi_3^*||\xi_4^*||\xi_5^*|.
\end{align}
\item[(3).] If $|\xi_4^*|\gg|\xi_5^*|$, then
\begin{align}\label{lem:Mk+2-3}
|M_{k+2}|\lesssim m(\xi_1^*)^2|\xi_1^*|^2|\xi_5^*|.
\end{align}
\end{itemize}
\end{lem}
Lemma \ref{es:Mk+2}  implies that the bound of $M_{k+2}$ in ``resonance'' set is less than its nature bound (which is $m(\xi_1^*)^2|\xi_1^*|^2|\xi_3^*|$). Here we give a slightly finer  bound than the one obtained in \cite{Miao-Shao-Wu-Xu:2009:gKdV}, according to the problem under study.


Based on these two lemmas, we rewrite the term \eqref{dE1} as
\begin{align*}
 \eqref{dE1}=&\int_{\Gamma_{k+2}} e^{i\alpha_{k+2}t}\Big(\chi_{\Omega}M_{k+2}+i\sigma_{k+2}\>\alpha_{k+2}\Big)
\widehat{f_\lambda}(t,\xi_1)\cdots\widehat{f_\lambda}(t,\xi_{k+2})
(d\xi_1)_\lambda\cdots (d\xi_{k+1})_\lambda\\
&\quad+\int_{\Gamma_{k+2}} e^{i\alpha_{k+2}t}\big(1-\chi_{\Omega}\big)M_{k+2}
\widehat{f_\lambda}(t,\xi_1)\cdots\widehat{f_\lambda}(t,\xi_{k+2})(d\xi_1)_\lambda\cdots (d\xi_{k+1})_\lambda.
\end{align*}
From Lemma \ref{non-res}, $\frac{1}{i\alpha_{k+2}} \big(\chi_{\Omega}M_{k+2}+i\sigma_{k+2}\>\alpha_{k+2}\big)$ is bounded. Thus by \eqref{identity}, integration by parts in time, \eqref{eqs:1719.47}, and combining with \eqref{dE2}, we have
\begin{align}
\dfrac{d}{dt}E(Iu_\lambda(t))=&\int_{\Gamma_{k+2}} \partial_t\big(e^{i\alpha_{k+2}t}\big)\frac{1}{i\alpha_{k+2}} \Big(\chi_{\Omega}M_{k+2}+i\sigma_{k+2}\>\alpha_{k+2}\Big)
\widehat{f_\lambda}(t,\xi_1)\cdots\widehat{f_\lambda}(t,\xi_{k+2})
\notag\\
&+\int_{\Gamma_{k+2}} e^{i\alpha_{k+2}t}\big(1-\chi_{\Omega}\big)M_{k+2}
\widehat{f_\lambda}(t,\xi_1)\cdots\widehat{f_\lambda}(t,\xi_{k+2})
\notag\\
&+\int_{\Gamma_{2k+2}} e^{i\alpha_{2k+2}t}M_{2k+2}\>
\widehat{f_\lambda}(t,\xi_1)\cdots\widehat{f_\lambda}(t,\xi_{2k+2})
\notag\\
=&\dfrac{d}{dt}\int_{\Gamma_{k+2}} e^{i\alpha_{k+2}t}\frac{\chi_{\Omega}M_{k+2}+i\sigma_{k+2}\>\alpha_{k+2}}{i\alpha_{k+2}}
\widehat{f_\lambda}(t,\xi_1)\cdots\widehat{f_\lambda}(t,\xi_{k+2})
\notag\\
&+\int_{\Gamma_{k+2}} e^{i\alpha_{k+2}t}\big(1-\chi_{\Omega}\big)M_{k+2}
\widehat{f_\lambda}(t,\xi_1)\cdots\widehat{f_\lambda}(t,\xi_{k+2})
\notag\\
&+\int_{\Gamma_{2k+2}} e^{i\alpha_{2k+2}t}\overline{M_{2k+2}}\>
\widehat{f_\lambda}(t,\xi_1)\cdots\widehat{f_\lambda}(t,\xi_{2k+2}),\label{3.10}
\end{align}
where
\begin{equation}\aligned
\overline{M_{2k+2}}:=
&i(k+2)\big[\tilde{\sigma}_{k+2}( \xi_1,\cdots, \xi_{k+1}, \xi_{{k+2}}+\cdots+
\xi_{2k+2}) ( \xi_{{k+2}}+\cdots+ \xi_{2k+2}\big)\big]_{sym},\label{M2k+2}
\endaligned
\end{equation}
and $\tilde{\sigma}_{k+2}=-{\chi_{\Omega}M_{k+2}}/{\alpha_{k+2}}$.  In particular, we have the bound of $\overline{M_{2k+2}}$.
\begin{lem}\label{es:M2k+2}
In $\Gamma_{2k+2}$,
\begin{itemize}
\item[(1).] It holds that
\begin{align}\label{M2k+2-1}
|\overline{M_{2k+2}}|\lesssim |\xi_1^*|.
\end{align}
\item[(2).] If $|\xi_1^*|\sim |\xi_2^*|\gtrsim N\gg |\xi_3^*|\sim |\xi_4^*|$, then
\begin{align}\label{M2k+2-2}
|\overline{M_{2k+2}}|\lesssim |\xi_3^*|.
\end{align}
\end{itemize}
\end{lem}

According to \eqref{3.10}, we define
\begin{align}\label{3.13}
E_I^2(u_\lambda(t)):= E(Iu_\lambda(t))-\int_{\Gamma_{k+2}} e^{i\alpha_{k+2}t}\frac{\chi_{\Omega}M_{k+2}+\sigma_{k+2}\>\alpha_{k+2}}{i\alpha_{k+2}}
\widehat{f_\lambda}(t,\xi_1)\cdots\widehat{f_\lambda}(t,\xi_{k+2}),
\end{align}
and get
\begin{align*}
\dfrac{d}{dt}E_I^2(u_\lambda(t))=&\int_{\Gamma_{k+2}} e^{i\alpha_{k+2}t}\big(1-\chi_{\Omega}\big)M_{k+2}
\widehat{f_\lambda}(t,\xi_1)\cdots\widehat{f_\lambda}(t,\xi_{k+2})
\\
&+\int_{\Gamma_{2k+2}} e^{i\alpha_{2k+2}}\overline{M_{2k+2}}\>
\widehat{f_\lambda}(t,\xi_1)\cdots\widehat{f_\lambda}(t,\xi_{2k+2}).
\end{align*}
This gives that 
\begin{align}
E_I^2(u_\lambda(t))=&E_I^2(u_\lambda(0))+\int_0^t\int_{\Gamma_{k+2}} e^{i\alpha_{k+2}s}\big(1-\chi_{\Omega}\big)M_{k+2}
\widehat{f_\lambda}(s,\xi_1)\cdots\widehat{f_\lambda}(s,\xi_{k+2})
\notag\\
&+\int_0^t\int_{\Gamma_{2k+2}} e^{i\alpha_{2k+2}s}\overline{M_{2k+2}}\>
\widehat{f_\lambda}(s,\xi_1)\cdots\widehat{f_\lambda}(s,\xi_{2k+2}).\label{eqs:EI2}
\end{align}

\texttt{Step 4: Energy increment estimates}.

By the preparation in Step 1--Step 3, we can derive the following proposition, which is sufficient to prove Theorem \ref{thm:main}.
\begin{prop}[Existence of an almost conserved quantity]\label{thm:main-2}
For a solution $u_\lambda$ to \eqref{gkdvR} which is smooth-in-time, Schwarz-in-space on the time interval $[0,\delta]$ with $\delta$ satisfying \eqref{delta}, we have
\begin{itemize}
\item (Fixed-time bound)
\begin{equation}\label{fixed-time bound}
\big|E_I^2(u_\lambda(t))-E(Iu_\lambda(t))\big|
\lesssim
N^{-2+}\|Iu_\lambda(t)\|^{k+2}_{H^1_x}.
\end{equation}
\item (Almost conservation law) Let $\|Iu_\lambda\|_{X^1([0,\delta])}\lesssim 1$, then
 \begin{equation}\label{Almost conserved}
 \left|E^2_I(u_\lambda(\delta))-E^2_I(u_\lambda(0))\right|\le K:= N^{-3+}+N^{-2+}\lambda^{-\frac12}.
 \end{equation}
\end{itemize}
\end{prop}

Now the paper is organized as follows. In Section 2, we introduce some
notations and state some preliminary estimates that will be used
throughout this paper. In particular, we establish the bilinear Strichartz estimates in this section. In Section 3, we prove Lemma \ref{non-res},  Lemma \ref{es:Mk+2} and Lemma \ref{es:M2k+2}. In Section 4,  we give the proof of Proposition \ref{thm:main-2}.  In Section 5, we show that Proposition \ref{thm:main-2} implies the global well-posedness
stated in Theorem 1.1.

\vspace{0.5cm}
\section{Notations and Preliminary Estimates}

\subsection{Basic notations and definitions}
We use $A\lesssim B$, $B\gtrsim A$, or sometimes $A=O(B)$ to denote the statement that $A\leq CB$ for some large constant $C$ which may
vary from line to line, and may depend on the data and the index $s$. When it is
necessary, we will write the constants by $C_1,C_2,\cdots$ to see
the dependency relationship. We use $A\sim B$ to mean
$A\lesssim B\lesssim A$.
We use $A\ll B$, or sometimes $A=o(B)$
to denote the statement $A\leq C^{-1}B$. The notation $a+$ denotes $a+\epsilon$ for
any small $\epsilon$, and $a-$ for $a-\epsilon$.
$\langle\cdot\rangle=(1+|\cdot|^2)^{1/2}$,
$D_x^\alpha=(-\partial^2_x)^{\alpha/2}$ and
$J_x^\alpha=(1-\partial^2_x)^{\alpha/2}$. We use
$\|f\|_{L^p_xL^q_t}$ to denote the mixed norm
$\Big(\displaystyle\int\|f(\cdot,x)\|_{L^q}^p\
dx\Big)^{\frac{1}{p}}$, and $\|f\|_{L^p_{xt}}:=\|f\|_{L^p_xL^p_t}$.

Throughout this paper, we use $\eta$ to denote a {\it smooth} cut-off function such that
\begin{align*}
\begin{cases}
\eta(x)=1, \ & |x|\le 1,\\
\eta(x)=0,    \ &|x|\ge 2.
\end{cases}
\end{align*}
For an interval $I\subset \R$, we denote $\chi_I$ as its characteristic function
\begin{align*}
\begin{cases}
\chi_I(x)=1, \ & x\in I,\\
\chi_I(x)=0,    \ &x\not\in I.
\end{cases}
\end{align*}

Now we introduce some other notations and definitions, some of which are employed from \cite{CKSTT-03-KDV}. We define
$(d\xi)_\lambda$ to be the normalized counting measure on
$\Z/\lambda$ such that
$$
\displaystyle\int
a(\xi)\,(d\xi)_\lambda=\frac{1}{\lambda}\sum\limits_{\xi\in
\frac\Z{\lambda}} a(\xi).
$$

The Fourier transform of a function $f$ on $\T_{\lambda}=\R/\lambda\Z$ is defined by
$$
\hat{f}(\xi)=\displaystyle\int_0^\lambda e^{-2\pi i  x\xi}f(x)\,dx,
$$
and thus the Fourier inversion formula
$$
f(x)=\displaystyle\int e^{2\pi i  x\xi} \hat{f}(\xi)\,(d\xi)_\lambda.
$$
Then the following usual properties of the Fourier transform hold,
\begin{eqnarray}
 &\|f\|_{L^2([0,\lambda])}
 = \big\|\hat{f}\big\|_{L^2((d \xi)_{\lambda})} \quad \mbox{(Plancherel)};\label{Plancherel}\\
 &\displaystyle\int_0^\lambda f(x)\overline{g(x)}\,dx
 = \displaystyle\int \hat{f}(\xi)\overline{\hat{g}(\xi)}\,(d\xi)_\lambda\quad \mbox{(Parseval)}\label{Parseval}; \\
 & \widehat{fg}(\xi)=\displaystyle\int
  \hat{f}(\xi-\xi_1)\hat{g}(\xi_1) \,(d\xi_1)_\lambda \quad \mbox{(Convolution)}. \label{(Convolution)}
\end{eqnarray}
We define the Sobolev space $H^s([0,\lambda])$ with the norm,
$$
\|f\|_{H^s([0,\lambda])}=\left\|\langle
\xi\rangle^s\hat{f}(\xi)\right\|_{L^2((d\xi)_\lambda)}.
$$

For $s,b\in \R $, define the Bourgain space $X_{s,b}$ to be the
closure of the Schwartz class under the norm
\begin{equation}
\|f\|_{X_{s,b}}:= \left(\int\!\!\!\!\int \langle
\xi\rangle^{2s}\langle\tau- \xi^3\rangle^{2b}|\hat{f}(
\tau,\xi)|^2\,(d\xi)_\lambda d\tau\right)^{\frac{1}{2}},\label{X}
\end{equation}
for any $\lambda$-periodic function $f$. The space $X_{s,\frac{1}{2}}$ barely fails to control the
$L^\infty([0,T], H^s(\R))$ norm. To rectify this we define the
slightly stronger space $Y^s$ under the norm
\begin{equation}\label{working space}
\|f\|_{Y^s} := \|f\|_{X_{s,\frac{1}{2}}}+\left\|\langle
\xi\rangle^s\hat{f}\right\|_{L^2((d\xi)_{\lambda}) L^1(d\tau)}.
\end{equation}
Moreover, we define the restricted space $Y^s(I)$ as
\begin{equation*}
\|f\|_{Y^s(I)}:= \inf\{\|\tilde f\|_{Y^s}:\tilde f=f, \mbox{ on } I\},
\end{equation*}
and as \eqref{workingspace},
$$
\|f\|_{X^s(I)}:=\|\mathcal Gf\|_{Y^s(I)}.
$$
If there is no confusion, we will not mention the restriction.
\subsection{Some linear estimates}

\begin{lem}\label{lem:non-smooth}
Let  $0<\delta <1$, and $f\in X_{0,\frac12}$. Then $\chi_{[0,\delta]}(t)f\in X_{0,b}$ for any $b<\frac12$, and
$$
\big\|\chi_{[0,\delta]}(t) f\big\|_{X_{0,b}}\lesssim \big\| f\big\|_{X_{0,\frac12}}.
$$
\end{lem}
\begin{proof}
Note that
$$
\big\|J^s_t\chi_{[0,\delta]}(t)\big\|_{L^p_t}+\big\|\chi_{[0,\delta]}(t)\big\|_{L^\infty_t}\lesssim 1
$$
for any $s<\frac1p$.  Then by the fractional product role, H\"older's and Sobolev's inequality, we have
\begin{align*}
\big\|\chi_{[0,\delta]}(t)e^{t\partial_{x}^3}f\big\|_{H^b_t}
&\lesssim
\big\|J_t^b\chi_{[0,\delta]}(t)\cdot e^{t\partial_{x}^3}f\big\|_{L^2_t}+
\big\|\chi_{[0,\delta]}(t)\cdot J_t^be^{t\partial_{x}^3}f\big\|_{L^2_t}\\
&\lesssim
\big\|J_t^b\chi_{[0,\delta]}(t)\big\|_{L^p_t} \big\|e^{t\partial_{x}^3}f\big\|_{L^q_t}
+\big\|\chi_{[0,\delta]}(t)\big\|_{L^\infty_t} \big\|J_t^be^{t\partial_{x}^3}f\big\|_{L^2_t}\\
&\lesssim
 \big\| f\big\|_{H^{\frac12}_t},
\end{align*}
where
$
p=2+,  b<\frac1p, \frac1p+\frac1q=\frac12.
$
By using this estimate, we have
\begin{align*}
\big\|\chi_{[0,\delta]}(t) f\big\|_{X_{0,b}}&=\big\|\chi_{[0,\delta]}(t)e^{t\partial_{x}^3}f(t,x)\big\|_{L^2_xH^b_t}\\
&\lesssim
 \big\| e^{t\partial_{x}^3} f\big\|_{L^2_xH^{\frac12}_t}
=
 \big\| f\big\|_{X_{0,\frac12}}.
\end{align*}
This proves  the lemma.
\end{proof}
Now we state some preliminary estimates which will be used in the
following sections. First we recall some well-known Strichartz
estimates (see \cite{Bourgain, CKSTT-04-gkdvT}, for examples):
\begin{equation}
              \|f\|_{L^4_{xt}}
   \lesssim
              \|f\|_{X_{0,\frac{1}{3}}},\label{XE1}
\end{equation}
and
\begin{equation}
              \|f\|_{L^6_{xt}}
   \lesssim
              \lambda^{0+}\|f\|_{X_{0+,\frac{1}{2}+}}.\label{XE2}
\end{equation}
It follows from the interpolation between (\ref{XE1}) and (\ref{XE2}) that
\begin{equation}
              \|f\|_{L^q_{xt}}
   \lesssim
              \lambda^{0+}\|f\|_{X_{0+,\frac{1}{2}-\sigma(q)}},\label{XE3}
\end{equation}
for all $4<q<6$ and $\sigma(q)<2(\frac{1}{q}-\frac{1}{6})$.

Since the $L^q_x$-norm is invariant under the gauge transformation, we have almost the same Strichartz estimates between $X^s$ and $Y^s$. In particular, we have the following two estimates.
\begin{lem}\label{lem:XE6-12}
\begin{itemize}
\item[(1).] Let $s>\frac12$ and $f\in X^s$, then
\begin{equation}
\|f\|_{L^\infty_{xt}} \lesssim
\|f\|_{X^s}.\label{XE4}
\end{equation}
\item[(2).]
Let $s>0$ and $f\in X^s$, then
\begin{equation}
              \|f\|_{L^6_{xt}}
   \lesssim
              \lambda^{0+}\|f\|_{X^{s}}.\label{XE5}
\end{equation}
\end{itemize}
\end{lem}
\begin{proof}
Let $g=\mathcal G f$.
For \eqref{XE4},
by Young's and Cauchy-Schwartz's inequalities, we have
$$
\|f\|_{L^\infty_{xt}}=\|g\|_{L^\infty_{xt}}\leq \left\|\hat{g}\right\|_{L^1((d\xi)_\lambda
d\tau)}\lesssim \left\|\langle
\xi\rangle^{\frac{1}{2}+}\hat g\right\|_{L^2((d\xi)_{\lambda})
L^1(d\tau)}.
$$

For \eqref{XE5}, by dyadic decomposition, we write
$f=\sum_{j=0}^\infty f_j$, for each dyadic constituents $f_j$ with
frequency support $\langle \xi \rangle\sim 2^j$. Then, by (\ref{XE3})
and (\ref{XE4}),
 \begin{equation*}
 \aligned
\|f\|_{L^6_{xt}}
& \leq \sum\limits_{j=0}^\infty
\|f_j\|_{L^6_{xt}}
\lesssim \sum\limits_{j=0}^\infty
\|f_j\|_{L^q_{xt}}^\theta\|f_j\|_{L^\infty_{xt}}^{1-\theta}\\
&\lesssim \lambda^{0+} \sum\limits_{j=0}^\infty
\|f_j\|_{X_{\epsilon,\frac{1}{2}}}^\theta\|f_j\|_{Y^\rho}^{1-\theta}
\lesssim \lambda^{0+}\sum\limits_{j=0}^\infty
2^{[\theta\epsilon+\rho(1-\theta)]j} \|f_j\|_{X_{0,\frac{1}{2}}}^\theta\|f_j\|_{Y^\rho}^{1-\theta}\\
&\lesssim \lambda^{0+}\sum\limits_{j=0}^\infty
2^{[\theta\epsilon+\rho(1-\theta)]j} \|f_j\|_{Y^0},
 \endaligned
 \end{equation*}
where $\rho>\frac{1}{2}$, and we choose $q=6-$ such that
$\epsilon=0+,\theta=1-$. Choosing $q$ close enough to 6 such that
$s>\theta\epsilon+\rho(1-\theta)$, then we have the claim by
Cauchy-Schwartz's inequality.
\end{proof}

By interpolating between (\ref{XE4}) and (\ref{XE5}), we have
\begin{equation}
              \|f\|_{L^q_{xt}}
   \lesssim
              \lambda^{0+}\|f\|_{Y^{\beta(q)}},\label{XE6}
\end{equation}
for all $6<q<\infty$ and $\beta(q)>(\frac{1}{2}-\frac{3}{q})$.

\subsection{Bilinear Strichartz estimate}
Now we present the bilinear Strichartz estimate of the periodic
version. Let $S_\lambda(t)$ be the solution map to the free
KdV equation
$$
\partial_t u +  \partial^3_{x} u=0, \quad \mbox{ in } [0,\lambda^3T]\times [0,\lambda],
$$
and the bilinear operator $I_M(f,g)$ satisfy
\begin{equation}\label{Bi operator}
\widehat{I_M(f,g)}(\xi)=\displaystyle\int_{\xi=\xi_1+\xi_2}
\chi_{\left\{|\xi_1^2-\xi_2^2|\gtrsim M\right\}}
\hat{f}(\xi_1)\hat{g}(\xi_2)\,(d \xi_1)_\lambda.
\end{equation}
First, we recall the
following result obtained in (7.29) in \cite{CKSTT-03-KDV},
\begin{lem}
Let $\phi_1,\phi_2$ be $\lambda-$periodic functions with both the
frequencies supported on $\{\xi:|\xi|\sim N\}$, then
\begin{equation}\label{Bi Str1}
    \left\|\eta^2(t)S_\lambda\phi_1 S_\lambda\phi_2\right\|_{L^2_{xt}}\leq
    \tilde{C}(N,\lambda)\|\phi_1\|_{L^2_x}\|\phi_2\|_{L^2_x},
\end{equation}
where
\begin{equation}\label{CNlambda1}
\tilde{C}(N,\lambda)=\biggl\{
\begin{array}{ll}
1,&N\leq 1,\\
(\frac{1}{\sqrt{N}}+\frac{1}{\lambda})^{\frac{1}{2}},&N>1.
\end{array}
\end{equation}
\end{lem}

The results above match Kato's smoothing effect in the real line
case. As a refinement, we give the following bilinear Strichartz
estimates.
\begin{prop}\label{TBi Str}
Let $\phi_1,\phi_2$ be $\lambda-$periodic functions, and the
operator $I_M$ be defined in (\ref{Bi operator}), then
\begin{equation}\label{Bi Str}
    \left\|\eta^2(t)I_M(S_\lambda\phi_1,S_\lambda\phi_2)\right\|_{L^2_{xt}}\leq
    C(M,\lambda)\|\phi_1\|_{L^2_x}\|\phi_2\|_{L^2_x},
\end{equation}
where
\begin{equation}\label{CNlambda}
C(M,\lambda)=\biggl\{
\begin{array}{ll}
1,&M\leq 1,\\
(\frac{1}{M}+\frac{1}{\lambda})^{\frac{1}{2}},&M>1.
\end{array}
\end{equation}
\end{prop}
\begin{proof}  When $M\lesssim 1$, it easily follows by
$L^4_{xt}L^4_{xt}$-H\"{o}lder and (\ref{XE1}). So we only consider
the case $M\gg1$. Then, by Plancherel's identity, the left-hand side (which denotes simply by LHS) of \eqref{Bi Str}  equals to
\begin{align*}
&\left\|\displaystyle\int_{\stackrel{\xi_1+\xi_2=\xi,} {
\tau_1+\tau_2=\tau}}\chi_{\{|\xi_1^2-\xi_2^2|\gtrsim
M\}}\hat{\eta}(\tau_1-\xi_1^3)\hat{\eta}(\tau_2-\xi_2^3)
\phi_1(\xi_1)\phi_2(\xi_2)\,(d\xi_1)_\lambda\,d\tau_1\right\|_{L^2((d\xi)_{\lambda}d\tau)}\\
=&
\left\|\displaystyle\int_{\xi_1+\xi_2=\xi}\chi_{\{|\xi_1^2-\xi_2^2|\gtrsim
M\}}\psi(\tau-\xi_1^3-\xi_2^3)
\phi_1(\xi_1)\phi_2(\xi_2)\,(d\xi_1)_\lambda\right\|_{L^2((d\xi)_{\lambda}d\tau)},
\end{align*}
where $\psi=\hat{\eta}\ast\hat{\eta}$. Then by H\"{o}lder's
inequality,
\begin{align*}
\mbox{LHS of }\eqref{Bi Str}\lesssim &
\Big\|\Big(\int_{\xi_1+\xi_2=\xi}\chi_{\{|\xi_1^2-\xi_2^2|\gtrsim
M\}}^2\psi(\tau-\xi_1^3-\xi_2^3) \,(d\xi_1)_\lambda\Big)^{\frac12}\\
&\quad \cdot
\Big(\int_{\xi_1+\xi_2=\xi} \psi(\tau-\xi_1^3-\xi_2^3)\phi_1(\xi_1)^2\phi_2(\xi_2)^2\,(d\xi_1)_\lambda\Big)^{\frac12}\Big\|_{L^2((d\xi)_{\lambda}d\tau)}\\
\lesssim &\left\|\displaystyle\int_{\xi_1+\xi_2=\xi}\chi_{\{|\xi_1^2-\xi_2^2|\gtrsim
M\}}\psi(\tau-\xi_1^3-\xi_2^3)\,(d\xi_1)_\lambda
\right\|_{L^\infty((d\xi)_{\lambda}d\tau)}^{\frac{1}{2}}\\
&\quad \cdot
\left\|\int_{\xi_1+\xi_2=\xi} \psi(\tau-\xi_1^3-\xi_2^3)\phi_1(\xi_1)^2\phi_2(\xi_2)^2\,(d\xi_1)_\lambda\right\|_{L^1((d\xi)_{\lambda}d\tau)}^{\frac12}\\
\lesssim &\left\|\displaystyle\int_{\xi_1+\xi_2=\xi}\chi_{\{|\xi_1^2-\xi_2^2|\gtrsim
M\}}\psi(\tau-\xi_1^3-\xi_2^3)\,(d\xi_1)_\lambda
\right\|_{L^\infty((d\xi)_{\lambda}d\tau)}^{\frac{1}{2}}\|\phi_1\|_{L^2}\|\phi_2\|_{L^2}.
\end{align*}
Therefore, we only need to show
\begin{equation}\label{2.17}
B:=\left\|\displaystyle\int_{\xi_1+\xi_2=\xi}\chi_{\{|\xi_1^2-\xi_2^2|\gtrsim
M\}}\psi(\tau-\xi_1^3-\xi_2^3)\,(d\xi_1)_\lambda
\right\|_{L^\infty((d\xi)_{\lambda}d\tau)} \lesssim C(M,\lambda)^2.
\end{equation}
Indeed, let the set
\begin{eqnarray*}
A_{\xi,\tau}=\{\xi_1\in \frac{1}{\lambda}\Z: \xi_2=\xi-\xi_1, |\xi_1^2-\xi_2^2|\gtrsim
M, \tau-\xi_1^3-\xi_2^3=O(1)\},
\end{eqnarray*}
then we have
\begin{equation}\label{Bound}
    B\lesssim \frac{1}{\lambda}\sup\limits_{\xi\in \frac{1}{\lambda}\Z,\tau\in\R}\#A_{\xi,\tau}.
\end{equation}
So it reduces to estimate $\sup\limits_{\xi\in \frac{1}{\lambda}\Z,\tau\in\R}\#A_{\xi,\tau}$. To this end, we rewrite $A_{\xi,\tau}$ as
\begin{eqnarray}
A_{\xi,\tau} &=& \{\xi_1\in \frac{1}{\lambda}\Z: |\xi||2\xi_1-\xi|\gtrsim  M,
(\xi_1-\frac{\xi}{2})^2=a+O(1/|\xi|)\}
\label{2.25}\\
&=& \{\xi_1\in \frac{1}{\lambda}\Z: \xi_2=\xi-\xi_1, |\xi_1^2-\xi_2^2|\gtrsim
M, \xi_1=\frac{1}{2}\xi\pm\sqrt{a+O(1/|\xi|)}\},\label{2.26}
\end{eqnarray}
where $a=\frac{\tau-\frac{1}{4}\xi^3}{3\xi}$. Here, we only consider the case $''+''$ in \eqref{2.26}, that is, $\xi_1> \frac12\xi$, and denote the corresponding set as $A_{\xi,\tau}^+$.
The other case $''-''$ is symmetrical. Now we show that
$A_{\xi,\tau}^+$ belongs to a set of length $\frac{1}{M}$. To this end, we consider the following
two cases separately:
$$
\textbf{Case 1}: |a| \lesssim 1/|\xi|; \qquad \textbf{Case 2}: |a| \gg 1/|\xi|.
$$
\textbf{Case 1:} $ |a| \lesssim 1/|\xi|$. By (\ref{2.25}), we have $(2\xi_1-\xi)^2\lesssim
1/|\xi|$ and thus
\begin{equation}\label{2.27}
M^2\lesssim (2\xi_1-\xi)^2\xi^2\lesssim |\xi|.
\end{equation}
Thus we get $(2\xi_1-\xi)^2\lesssim
1/M^2$, that is,
$$
\xi_1=\frac{1}{2}\xi+O(\frac 1M).
$$
This implies that $A_{\xi,\tau}$
belongs to a set of length $\frac{1}{M}$.

\noindent\textbf{Case 2:}$|a| \gg 1/|\xi|.$ By (\ref{2.25}) again one has
$(2\xi_1-\xi)^2\sim a$ and thus
$$
a\gtrsim \frac{M^2}{\xi^2}.
$$
For any $x_1,x_2\in A_{\xi,\tau}^+$, by (\ref{2.26}) we find
\begin{eqnarray*}
|x_1-x_2|&=&\left|\sqrt{a+\varepsilon_1}-\sqrt{a+\varepsilon_2}\right|\\
&=&\frac{|\varepsilon_1-\varepsilon_2|}{\sqrt{a+\varepsilon_1}+\sqrt{a+\varepsilon_2}}\lesssim
\frac{1/|\xi|}{M/|\xi|}=
\frac{1}{M},
\end{eqnarray*}
where $\varepsilon_1,\varepsilon_2=O(1/|\xi|)$ and we have used the restriction $|a| \gg 1/|\xi|$. This also implies that $A_{\xi,\tau}^+$
belongs to a set of length $\frac{1}{M}$.
Therefore, no matter what case, we have
$$
\#A_{\xi,\tau}\lesssim \frac{\lambda}{M}+1.
$$
Then by (\ref{2.17}) and (\ref{Bound}), we have the claim.
\end{proof}

This proposition implies
\begin{cor}Let $u=u(t,x),v=v(t,x)$ be the $\lambda-$periodic functions of $x$,
then
\begin{equation}\label{EIN}
    \left\|\eta^2(t)I_M(u,v)\right\|_{L^2_{xt}}
    \lesssim
    C(M,\lambda)\|u\|_{X_{0,\frac{1}{2}+}}\|v\|_{X_{0,\frac{1}{2}+}}.
\end{equation}
\end{cor}

\begin{remark}\label{rem:Bi-Str}
In particular, we set $\lambda$ to be the number in \eqref{lambda}. Then we see that $C(N^2,\lambda)$, for which bound
we use in this paper, has the similar size of $\lambda^{-\frac12}$ rather than $N^{-1}$. Indeed, when $s\ge \frac12$, $k=3,4$,
$$
(1-s)/(\frac{2}{k}+s-\frac{1}{2})<2,
$$
thus $\lambda^{-\frac12}>N^{-1}$.
This means that the efficacy of the bilinear Strichartz estimate in the periodic case is exactly weaker than the one in the real line case.
\end{remark}

\begin{cor}\label{cor:2.2} Let $u,v,I_M$ be as Corollary 2.1, and let $\lambda$ be the number in \eqref{lambda},
then for $N\gg 1$,
\begin{equation}\label{CEIN}
    \left\|\eta^2(t)I_{N^2}(u,v)\right\|_{L^2_{xt}}
    \lesssim
    \lambda^{-\frac12+}\|u\|_{X_{0,\frac{1}{2}-}}\|v\|_{X_{0,\frac{1}{2}-}}.
\end{equation}
\end{cor}
\begin{proof}
First, by interpolating between (\ref{EIN}) and the following estimate
$$
\left\|\eta^2(t)I_M(u,v)\right\|_{L^2_{xt}}
    \lesssim \|u\|_{L^4_{xt}}\|v\|_{L^4_{xt}}
    \lesssim \|u\|_{X_{0,\frac{1}{3}}}\|v\|_{X_{0,\frac{1}{3}}},
$$
we have
\begin{equation}\label{15.05}
    \left\|\eta^2(t)I_M(u,v)\right\|_{L^2_{xt}}
    \lesssim
    C(M,\lambda)^{1-}\|u\|_{X_{0,\frac{1}{2}-}}\|v\|_{X_{0,\frac{1}{2}-}}.
\end{equation}
In particular, when $M=N^2$, by Remark \ref{rem:Bi-Str},
$$
C(N^2,\lambda)=\lambda^{-\frac12},\quad \mbox{whenever } N\gg 1.
$$
This proves the corollary.
\end{proof}

\section{Proof of Lemmas \ref{non-res}--\ref{es:M2k+2}}

\subsection{Proof of Lemma \ref{non-res}}
Note that
$$
\Omega=\Omega_1\cup \Omega_2\cup \Omega_3\cup \Omega_4,
$$
so we need to prove that in every $\Omega_j, j=1,2,3,4$,
$$
|M_{k+2}|\lesssim |\alpha_{k+2}|.
$$
The estimates in $\Omega_1, \Omega_2$ and $\Omega_4$ are almost the same as  Lemma 4.2 in \cite{Miao-Shao-Wu-Xu:2009:gKdV}, however, as one of the key lemmas in this paper, we still give a detail proof here for the sake of completeness. To simplify the notations, we set $\xi_j^*=\xi_j, j=1,\cdots, k+2$.

In $\Omega_1$, we note that $\xi_1\cdot\xi_2<0$, thus,
\begin{align*}
|\alpha_{k+2}|=&\big|\xi_1^3+\xi_2^3+\xi_3^3\big|+o(|\xi_3^3|)\\
= & \big|(\xi_1+\xi_2)(\xi_1^2-\xi_1\xi_2+\xi_2^2)+\xi_3^3\big|+o(|\xi_3^3|)\\
= & \big|\xi_3(\xi_1^2-\xi_1\xi_2+\xi_2^2-\xi_3^2)\big|+o(|\xi_3^3|)
\ge  \big|\xi_3\xi_1^2\big|+o(|\xi_3^3|)\sim |\xi_3||\xi_1|^2.
\end{align*}
Moreover, by the mean value theorem,
\begin{align*}
\big|M_{k+2}\big|
\lesssim &
\big|m_1^2\xi_1^3+m_2^2\xi_2^3\big|+\big|m_3^2\xi_3^3\big|+\cdots+\big|m_{k+2}^2\xi_{k+2}^3\big|\\
\lesssim &
m_1^2\big|\xi_1+\xi_2\big|\xi_1^2+\big|\xi_3^3\big|\\
\lesssim &
 \big|\xi_3\big|\xi_1^2.
\end{align*}
Thus we obtain the desirable estimates in $\Omega_1$.

In $\Omega_2$, we have
$$
|\alpha_{k+2}|\sim  \big|\xi_1^3+\xi_2^3\big|.
$$
Moreover,
\begin{align*}
\big|M_{k+2}\big|
\le &
\big|m_1^2\xi_1^3+m_2^2\xi_2^3\big|+\big|\xi_3^3+\cdots+m_{k+2}^2\xi_{k+2}^3\big|\\
\lesssim &
m_1^2\big|\xi_1^3+\xi_2^3\big|+\big|\xi_3^3+\cdots+m_{k+2}^2\xi_{k+2}^3\big|
\lesssim
\big|\xi_1^3+\xi_2^3\big|.
\end{align*}
So these give the desirable estimates in $\Omega_2$.

In $\Omega_3$,
on one hand, since
\begin{equation}\label{1216.53}
\big|\xi_1^3+\xi_2^3\big|\sim \xi_1^2|\xi_1+\xi_2|\gg |\xi_1||\xi_3|^2\gg |\xi_3|^3,
\end{equation}
thus,
$$
|\alpha_{k+2}|= \big|(\xi_1^3+\xi_2^3)+(\xi_3^3+\cdots+\xi_{k+2}^3)\big|\sim \big|\xi_1^3+\xi_2^3\big|.
$$
On the other hand, by the mean value theorem and \eqref{1216.53},
\begin{align*}
\big|M_{k+2}\big|
\lesssim &
\big|m_1^2\xi_1^3+m_2^2\xi_2^3\big|+\big|m_3^2\xi_3^3\big|+\cdots+\big|m_{k+2}^2\xi_{k+2}^3\big|\\
\lesssim &
m_1^2\big|\xi_1^3+\xi_2^3\big|+\big|\xi_3^3\big|\\
\lesssim &
 \big|\xi_1^3+\xi_2^3\big|.
\end{align*}
Combining these two estimates, gives the desirable estimates in $\Omega_3$.

For $\Omega_4$, we set $\overline{\xi_4}=\xi_4+\xi_5+\xi_6$, then $\xi_1+\xi_2+\xi_3+\overline{\xi_4}=0$.
Therefore,
\begin{align*}
\alpha_{k+2}=&\xi_1^3+\xi_2^3+\xi_3^3+\overline{\xi_4}^3+\big(\xi_4^3-\overline{\xi_4}^3\big)+\xi_5^3+\xi_6^3\\
=&3(\xi_1+\xi_2)(\xi_1+\xi_3)(\xi_1+\overline{\xi_4})+\big(\xi_4^3-\overline{\xi_4}^3\big)+\xi_5^3+\xi_6^3\\
= &3(\xi_1+\xi_2)(\xi_1+\xi_3)(\xi_1+\xi_4)+O(|\xi_5+\xi_6|\xi_1^2)\\
= &3(\xi_1+\xi_2)(\xi_1+\xi_3)(\xi_1+\xi_4)+O(|\xi_5|\xi_1^2).
\end{align*}
By the definition of $\Omega_4$, $ |\xi_1+\xi_2||\xi_1+\xi_3||\xi_1+\xi_4|\gg|\xi_5|\xi_1^2$. Thus we have
\begin{align}\label{17.26}
|\alpha_{k+2}|\sim |\xi_1+\xi_2||\xi_1+\xi_3||\xi_1+\xi_4|.
\end{align}
By the similar way and the mean value theorem, we have
\begin{align}\label{17.27}
M_{k+2}=&m(\xi_1)^2\xi_1^3+m(\xi_2)^2\xi_2^3+m(\xi_3)^2\xi_3^3+m(\overline{\xi_4})^2\overline{\xi_4}^3
+\big(m(\xi_4)\xi_4^3-m(\overline{\xi_4})\overline{\xi_4}^3\big)+\xi_5^3+\xi_6^3\notag\\
=&m(\xi_1)^2\xi_1^3+m(\xi_2)^2\xi_2^3+m(\xi_3)^2\xi_3^3+m(\overline{\xi_4})^2\overline{\xi_4}^3+O(|\xi_5|\xi_1^2).
\end{align}
Now we claim that
\begin{align}\label{17.28}
\big|m(\xi_1)^2\xi_1^3+m(\xi_2)^2\xi_2^3+m(\xi_3)^2\xi_3^3+m(\overline{\xi_4})^2\overline{\xi_4}^3\big|
\lesssim
|\xi_1+\xi_2||\xi_1+\xi_3||\xi_1+\xi_4|.
\end{align}
To prove this, we split it into two cases: either $|\xi_1|-|\xi_4|=o(|\xi_1|)$, or $|\xi_1|-|\xi_4|\sim |\xi_1|$.
The first case follows from the following double mean value theorem.
\begin{lem}[Double mean value theorem, Lemma 4.1 in \cite{Miao-Shao-Wu-Xu:2009:gKdV}]\label{lem:DMVT}
Let $f(\xi)=m(\xi)^2\xi^3$, then for $|\eta|,|\lambda|\ll |\xi|$,
\begin{equation}\label{DMVT}
\left| f(\xi+\eta+\lambda)-f(\xi+\eta)-f(\xi+\lambda)+f(\xi)\right|
\lesssim \big|f''(\xi)\big||\eta||\lambda| .
\end{equation}
\end{lem}
Indeed, using \eqref{DMVT}, we have
\begin{align*}
\big|m(\xi_1)^2\xi_1^3+m(\xi_2)^2\xi_2^3&+m(\xi_3)^2\xi_3^3+m(\overline{\xi_4})^2\overline{\xi_4}^3\big|
\lesssim
m_1^2|\xi_1+\xi_2||\xi_1+\xi_3||\xi_1+\overline{\xi_4}|\\
\lesssim&
|\xi_1+\xi_2||\xi_1+\xi_3||\xi_1+\xi_4|+O(|\xi_5|\xi_1^2)
\sim
|\xi_1+\xi_2||\xi_1+\xi_3||\xi_1+\xi_4|.
\end{align*}
In the second case, we also have $|\xi_1|-|\xi_3|\sim |\xi_1|$. Thus, it gives that
$$
|\xi_1+\xi_2||\xi_1+\xi_3||\xi_1+\xi_4|\sim\xi_1^2|\xi_1+\xi_2|.
$$
Therefore,
\begin{align*}
\big|m(\xi_1)^2\xi_1^3+m(\xi_2)^2\xi_2^3&+m(\xi_3)^2\xi_3^3+m(\overline{\xi_4})^2\overline{\xi_4}^3\big|\\
\lesssim&
\big|m(\xi_1)^2\xi_1^3+m(\xi_2)^2\xi_2^3\big|+\big|m(\xi_3)^2\xi_3^3+m(\overline{\xi_4})^2\overline{\xi_4}^3\big|\\
\lesssim&
m_1^2\xi_1^2|\xi_1+\xi_2|+m_3^2\xi_3^2|\xi_3+\xi_4|
\lesssim
\xi_1^2|\xi_1+\xi_2|+\xi_1^2|\xi_5|\\
\sim&
|\xi_1+\xi_2||\xi_1+\xi_3||\xi_1+\xi_4|.
\end{align*}
This proves \eqref{17.28}. Now combining with \eqref{17.27}, we have
\begin{align*}
|M_{k+2}|\lesssim |\xi_1+\xi_2||\xi_1+\xi_3||\xi_1+\xi_4|+O(|\xi_5|\xi_1^2)
\sim  |\xi_1+\xi_2||\xi_1+\xi_3||\xi_1+\xi_4|.
\end{align*}
Together with \eqref{17.26}, we find $|M_{k+2}|\lesssim |\alpha_{k+2}|$, which is the desirable estimate in $\Omega_4$.
This completes the proof of the lemma.

\subsection{Proof of Lemma \ref{es:Mk+2}}
We may assume that $|\xi_1|\ge \cdots \ge |\xi_{k+2}|$ by symmetries, and set $\xi_6=0$ if $k=3$. Recall that
$$
\Gamma_{k+2}\backslash \Omega=\big(\Gamma_{k+2}\backslash \Omega_1\big)\cap \big(\Gamma_{k+2}\backslash \Omega_2\big)\cap \big(\Gamma_{k+2}\backslash \Omega_3\big)\cap \big(\Gamma_{k+2}\backslash \Omega_4\big).
$$

First, we consider \eqref{lem:Mk+2-1}. If $|\xi_1|\sim |\xi_2|\sim |\xi_3|$, then
\begin{align*}
\big|M_{k+2}\big|
\lesssim
\big|m_1^2\xi_1^3\big|
\sim
m_1^2|\xi_1||\xi_3|^2.
\end{align*}
If $|\xi_1|\sim |\xi_2|\gg |\xi_3|$, then by the definition of $\Omega_1$ and $\Omega_3$, we have in $\Gamma_{k+2}\backslash \Omega$,
$$
|\xi_1|\sim |\xi_2|\gg |\xi_3|\sim |\xi_4|,\quad \mbox{ and } \quad |\xi_1|\big|\xi_1+\xi_2|\big|\lesssim |\xi_3|^2.
$$
Then by the mean value theorem and the inequality $m(\xi)^2|\xi|\le m(\eta)^2|\eta|$ if $|\xi|\le |\eta|$, we have
\begin{align*}
\big|M_{k+2}\big|
\le &
\big|m_1^2\xi_1^3+m_2^2\xi_2^3\big|+\big|m_3^2\xi_3^3+\cdots+m_{k+2}^2\xi_{k+2}^3\big|\\
\lesssim &
m_1^2\xi_1^2\big|\xi_1+\xi_2\big|+m_3^2|\xi_3|^3\\
\lesssim &
m_1^2|\xi_1||\xi_3|^2+m_3^2|\xi_3|^3\\
\lesssim &
m_1^2|\xi_1||\xi_3|^2.
\end{align*}
This proves  \eqref{lem:Mk+2-1}.

Now we consider \eqref{lem:Mk+2-2}. By the definition of $\Omega_2$, we have $\big|\xi_1^3+\xi_2^3\big|\lesssim \big|\xi_3^3+\cdots+\xi_{k+2}^3\big|$ in $\Gamma_{k+2}\backslash \Omega_2$.  Then by the mean value theorem, we have
\begin{align*}
\big|M_{k+2}\big|
\le &
\big|m_1^2\xi_1^3+m_2^2\xi_2^3\big|+\big|\xi_3^3+\cdots+\xi_{k+2}^3\big|\\
\lesssim &
m_1^2\big|\xi_1^3+\xi_2^3\big|+\big|\xi_3^3+\cdots+\xi_{k+2}^3\big|\\
\lesssim &
\big|\xi_3^3+\cdots+\xi_{k+2}^3\big|\\
\lesssim &
\big|\xi_3\big|\big|\xi_4\big|\big|\xi_5\big|.
\end{align*}

We turn to consider \eqref{lem:Mk+2-3}.  According to the definition of $\Omega_3$, we split it into the following two subsets,
\begin{equation*}
\aligned
  A_{1} =&\; \{(\xi_1,\cdots,\xi_6)\in \Gamma_{k+2}\backslash \Omega: |\xi_4|\gg|\xi_5|,
  |m^2_1\xi_1^3+\cdots+m^2_4 \xi_4^3|\lesssim\big|m_5^2\xi_5^3+m_6^2\xi_6^3\big|\}; \\
  A_{2} =&\; \{(\xi_1,\cdots,\xi_6)\in \Gamma_{k+2}\backslash \Omega: |\xi_4|\gg|\xi_5|,
  |( \xi_1+ \xi_2)( \xi_1+ \xi_3)( \xi_1+ \xi_4)|\lesssim
| \xi_1|^2| \xi_5|\}.
\endaligned
\end{equation*}

In $A_{1}$, we have
\begin{align*}
\big|M_{k+2}\big|
\le &
\big|m^2_1\xi_1^3+\cdots+m^2_4 \xi_4^3\big|+\big|m_5^2\xi_5^3+m_6^2\xi_6^3\big|\\
\lesssim &
m_5^2|\xi_5|^3
\lesssim
m_1^2\xi_1^2|\xi_5|.
\end{align*}

In $A_{2}$, we may assume that $|\xi_1+\xi_3|\ll |\xi_1|$ or $|\xi_1+\xi_4|\ll |\xi_1|$. Otherwise, if $|\xi_1+\xi_3|\gtrsim |\xi_1|$ and $|\xi_1+\xi_4|\gtrsim |\xi_1|$, then from the relation in $A_{3}$, we have  $| \xi_1+ \xi_2|\lesssim |\xi_5|$, which is included in $A_{1}$.
Therefore, by the definition of $\Omega_1$, we have $|\xi_1|\sim |\xi_2|\sim |\xi_3|\sim |\xi_4|$ in $A_3$. Further, we set $\xi_1>0$ by symmetries, and then have three cases as follows,
\begin{align*}
(1), \xi_1>0, \xi_2<0,&\,\xi_3<0,\xi_4>0;\quad
(2), \xi_1>0, \xi_2<0,\xi_3>0,\xi_4<0;\\
&(3), \xi_1>0, \xi_2>0,\xi_3<0,\xi_4<0.
\end{align*}

For (1), we take $\xi=\xi_1,\eta=-(\xi_1+\xi_2), \lambda=-(\xi_1+\xi_3)$, then $|\eta|\lesssim |\lambda|\ll |\xi|$. Using Lemma \ref{lem:DMVT},  we have
\begin{align*}
|m^2_1\xi_1^3+\cdots+m^2_4 \xi_4^3|\lesssim m_1^2|\xi_1+\xi_2||\xi_1+\xi_3||\xi_1+\xi_4|
\lesssim m_1^2|\xi_1|^2|\xi_5|,
\end{align*}
where we have used $ \big|f''(\xi_1)\big|\sim m^2_1|\xi_1| $.
Thus,
\begin{align*}
\big|M_{k+2}\big|
\le &
\big|m^2_1\xi_1^3+\cdots+m^2_4 \xi_4^3\big|+\big|m_5^2\xi_5^3+m_6^2\xi_6^3\big|\\
\lesssim &
m_1^2|\xi_1+\xi_2||\xi_1+\xi_3||\xi_1+\xi_4|+\big|m_5^2\xi_5^3+m_6^2\xi_6^3\big|\\
\lesssim &
m_1^2|\xi_1|^2|\xi_5|+m_5^2|\xi_5|^3
\lesssim
m_1^2\xi_1^2|\xi_5|.
\end{align*}

For (2), we take $\xi=\xi_1,\eta=-(\xi_1+\xi_2), \lambda=-(\xi_1+\xi_4)$; For (3), we take $\xi=\xi_1,\eta=-(\xi_1+\xi_3), \lambda=-(\xi_1+\xi_4)$. Then by the same argument, we get the desired estimates.
This proves the lemma.

\subsection{Proof of Lemma \ref{es:M2k+2}} 
The proof of this lemma is essentially presented in Lemma 4.4
in \cite{Miao-Shao-Wu-Xu:2009:gKdV}, so we only present the sketch of the proof here. 
Again, we assume that $|\xi_1|\ge \cdots \ge |\xi_{k+2}|$ by symmetries. First, from Lemma \ref{non-res}, we have $|\tilde{\sigma}_{k+2}|\lesssim 1$. Therefore, $|M_{2k+2}|\lesssim |\xi_1|$.
Now we prove \eqref{M2k+2-2}. From \eqref{M2k+2}, we rewrite $M_{2k+2}$ as
\begin{align*}
M_{2k+2}=& i\tilde{\sigma}_{k+2}( \overline{\xi_1},\xi_{k+2}, \xi_{k+3},\cdots, \xi_{2k+2})  \overline{\xi_1}
+i\tilde{\sigma}_{k+2}(\xi_1, \overline{\xi_2}, \xi_{k+3}, \xi_{k+4},\cdots, \xi_{2k+2})  \overline{\xi_2}\\
&+i\tilde{\sigma}_{k+2}(\xi_1, \xi_2, \overline{\xi_3}, \xi_{k+4}, \xi_{k+5},\cdots, \xi_{2k+2})  \overline{\xi_3} +\cdots+i\tilde{\sigma}_{k+2}(\xi_1,  \xi_2, \xi_3, \cdots, \xi_{k+1}, \overline{\xi_{k+2}})  \overline{\xi_{k+2}},
\end{align*}
where $\overline{\xi_j}=\xi_j+\cdots+\xi_{k+j}$. Since $|\overline{\xi_j}|\lesssim |\xi_j|\le |\xi_3|$ for any $j=3,\cdots, k+2$, and $|\tilde{\sigma}_{k+2}|\lesssim 1$, we have
\begin{align*}
&\big|i\tilde{\sigma}_{k+2}(\xi_1, \xi_2, \overline{\xi_3}, \xi_{k+4}, \xi_{k+5},\cdots, \xi_{2k+2})  \overline{\xi_3} +\cdots+i\tilde{\sigma}_{k+2}(\xi_1,  \xi_2, \xi_3, \cdots, \xi_{k+1}, \overline{\xi_{k+2}})  \overline{\xi_{k+2}}\big|\\
\lesssim & |\overline{\xi_3}|+\cdots+|\overline{\xi_{k+2}}|\lesssim |\xi_3|.
\end{align*}
Furthermore, by the mean value theorem, 
\begin{align*}
&\big|i\tilde{\sigma}_{k+2}( \overline{\xi_1},\xi_{k+2}, \xi_{k+3},\cdots, \xi_{2k+2})  \overline{\xi_1}
+i\tilde{\sigma}_{k+2}(\xi_1, \overline{\xi_2}, \xi_{k+3}, \xi_{k+4},\cdots, \xi_{2k+2})  \overline{\xi_2}\big|\\
\lesssim & |\overline{\xi_1}+\overline{\xi_{2}}|\lesssim |\xi_3|.
\end{align*}
This proves the lemma.
\section{Proof of Proposition \ref{thm:main-2}}

We first give the fixed-time bound in Proposition \ref{thm:main-2}. By \eqref{3.13}, it reduces to the following lemma.
\begin{lem}\label{lem:fixed-time}
For any $1/2\leq s<1$,
\begin{equation*}
\left|\int_{\Gamma_{k+2}} e^{i\alpha_{k+2}t}\frac{\chi_{\Omega}M_{k+2}+\sigma_{k+2}\>\alpha_{k+2}}{i\alpha_{k+2}}
\widehat{f_\lambda}(t,\xi_1)\cdots\widehat{f_\lambda}(t,\xi_{k+2})\right| \lesssim
N^{-2+}\|Iu_\lambda(t)\|^{k+2}_{H^1_x}.
\end{equation*}
\end{lem}
\begin{proof}
First, we assume that $\widehat {u_\lambda}$ is positive, otherwise one may replace it by $|\widehat{u_\lambda}|$. Second, we also assume that $|\xi_1|\ge \cdots \ge |\xi_{k+2}|$ by symmetries. Moreover, by the reduction in Remark \ref{rem:reduction}, we further assume $|\xi_1|\sim |\xi_2|\gtrsim N$.
Now by Lemma \ref{non-res}, we have
$$
|\frac{\chi_{\Omega}M_{k+2}+\sigma_{k+2}\>\alpha_{k+2}}{i\alpha_{k+2}}|\lesssim 1.
$$
Therefore, by H\"older's inequality and Sobolev's inequality, we have
\begin{align*}
&\left|\int_{\Gamma_{k+2}} e^{i\alpha_{k+2}t}\frac{\chi_{\Omega}M_{k+2}+\sigma_{k+2}\>\alpha_{k+2}}{i\alpha_{k+2}}
\widehat{f_\lambda}(t,\xi_1)\cdots\widehat{f_\lambda}(t,\xi_{k+2})\right| \\
=&
\left|\int_{\Gamma_{k+2}} \frac{\chi_{\Omega}M_{k+2}+\sigma_{k+2}\>\alpha_{k+2}}{i\alpha_{k+2}}
\widehat{u_\lambda}(t,\xi_1)\cdots\widehat{u_\lambda}(t,\xi_{k+2})\right| \\
\lesssim &
\left|\int_{\Gamma_{k+2}}
\widehat{u_\lambda}(t,\xi_1)\cdots\widehat{u_\lambda}(t,\xi_{k+2})\right| \\
=&
N^{2(s-1)}\Big|\int_{\Gamma_{k+2}}|\xi_1|^{-s+k\epsilon}|\xi_2|^{-s}
\widehat{|\nabla| Iu_\lambda}(t,\xi_1)\widehat{|\nabla| Iu_\lambda}(t,\xi_{2})\widehat{|\nabla|^{-\epsilon} u_\lambda}(t,\xi_{3})\cdots
\widehat{|\nabla|^{-\epsilon} u_\lambda}(t,\xi_{k+2})
\Big|\\
\lesssim &
N^{-2+k\epsilon}\Big|\int
{|\nabla| Iu_\lambda}(t,x){|\nabla| Iu_\lambda}(t,x)
\big({|\nabla|^{-\epsilon} Iu_\lambda}(t,x)\big)^{k}
\Big|\\
\lesssim &
N^{-2+k\epsilon}\big\|Iu_\lambda\big\|_{L^\infty_tH^1_x}^2\big\||\nabla|^{-\epsilon} u_\lambda\big\|_{L^\infty_{xt}}^k\\
\lesssim &
N^{-2+k\epsilon}\big\|Iu_\lambda\big\|_{L^\infty_tH^1_x}^{k+2},
\end{align*}
where we have used that for any $\frac12\le s \le 1$,
$$
\big\||\nabla|^{-\epsilon} u_\lambda\big\|_{L^\infty_{xt}}\lesssim
\big\| u_\lambda\big\|_{L^\infty_{t}H^\frac12_x}
\lesssim
\big\| Iu_\lambda\big\|_{L^\infty_{t}H^1_x}.
$$
This proves the lemma.
\end{proof}

Lemma \ref{lem:fixed-time} proves \eqref{fixed-time bound}. To prove \eqref{Almost conserved}, by \eqref{eqs:EI2}, we need  to estimate
\begin{align*}
\int_0^\delta\int_{\Gamma_{k+2}} e^{i\alpha_{k+2}s}\big(1-\chi_{\Omega}\big)M_{k+2}
\widehat{f_\lambda}(s,\xi_1)\cdots\widehat{f_\lambda}(s,\xi_{k+2}),
\end{align*}
and
\begin{align*}
&\int_0^\delta\int_{\Gamma_{2k+2}} e^{i\alpha_{2k+2}s}\overline{M_{2k+2}}\>
\widehat{f_\lambda}(s,\xi_1)\cdots\widehat{f_\lambda}(s,\xi_{2k+2}).
\end{align*}
These are included in the following two lemmas.

\begin{lem} \label{p+2-linear}
Let $s\geq\frac{1}{2}$, and $ \|Iu_\lambda\|_{X^1([0,\delta])}\lesssim  1$, then
\begin{equation}
 \left|\int_0^\delta\!\!\int_{\Gamma_{k+2}} e^{i\alpha_{k+2}s}\big(1-\chi_{\Omega}\big)M_{k+2}
\widehat{f_\lambda}(s,\xi_1)\cdots\widehat{f_\lambda}(s,\xi_{k+2})\right|
 \lesssim
 K,
 \label{Lambda6}
\end{equation}
where $K= N^{-3+}+N^{-2+}\lambda^{-\frac12}$.
\end{lem}
\begin{proof}
Before estimation, we give several reductions.
First, let $v=\mathcal G u$, then
$$
\hat v(t,\xi)=e^{2\pi i \xi\int_0^t \!\!\int_\T u^k\,dxds} \hat u(t,\xi).
$$
So for $\xi_1+\cdots+\xi_{k+2}=0$,
$$
\hat u(t,\xi_1)\cdots \hat u(t,\xi_{k+2})=\hat v(t,\xi_1)\cdots \hat v(t,\xi_{k+2}).
$$
After rescaling, this gives that
$$
\hat u_\lambda(t,\xi_1)\cdots \hat u_\lambda(t,\xi_{k+2})=\hat v_\lambda(t,\xi_1)\cdots \hat v_\lambda(t,\xi_{k+2}).
$$
Thus,
\begin{align*}
 &\int_0^\delta\!\!\int_{\Gamma_{k+2}} e^{i\alpha_{k+2}s}\big(1-\chi_{\Omega}\big)M_{k+2}
\widehat{f_\lambda}(s,\xi_1)\cdots\widehat{f_\lambda}(s,\xi_{k+2})\\
 =&
 \int_0^\delta\!\!\int_{\Gamma_{k+2}} \big(1-\chi_{\Omega}\big)M_{k+2}
\widehat{u_\lambda}(s,\xi_1)\cdots\widehat{u_\lambda}(s,\xi_{k+2})\\
 =&
 \int_0^\delta\!\!\int_{\Gamma_{k+2}} \big(1-\chi_{\Omega}\big)M_{k+2}
\widehat{v_\lambda}(s,\xi_1)\cdots\widehat{v_\lambda}(s,\xi_{k+2}).
\end{align*}

Second, to extend the integration domain from $[0,\delta]$ to $\R$, we insert the non-smooth cutoff function $\chi_{[0,\delta]}(t)$ into one of $v_\lambda$ and use the estimate in Lemma \ref{lem:non-smooth}. This allows us to turn to show
$$
 \left|\int_\R \!\int_{\Gamma_{k+2}} \big(1-\chi_{\Omega}\big)M_{k+2}
\widehat{v_\lambda}(s,\xi_1)\cdots\widehat{v_\lambda}(s,\xi_{k+2})\right|
 \lesssim
 K\>
 \|Iv_\lambda\|_{X_{1,\frac{1}{2}-}}\|Iv_\lambda\|_{Y^1}^{k+1},
$$
where $v_\lambda$ is time supported on $[0,\delta]$.
But the $0+$ loss is not essential and will be recorded by $N^{0+}$, thus it will not be mentioned.
Then by Plancherel's identity, it turns to show
\begin{equation}\label{3.11}
 \left|\iint_{\Gamma_{k+2}\times \Gamma_{k+2}} \big(1-\chi_{\Omega}\big)M_{k+2}
\widehat{v_\lambda}(\tau_1,\xi_1)\cdots\widehat{v_\lambda}(\tau_{k+2},\xi_{k+2})\right|
\lesssim K\>\|Iv_\lambda\|_{Y^1}^{k+2},
\end{equation}
where the set $\Gamma_{k+2}\times
\Gamma_{k+2}=\{( \xi,\tau): \xi_1+\cdots+ \xi_{k+2}=0,
\tau_1+\cdots+\tau_{k+2}=0\}$ and we write $ \xi=( \xi_1,\cdots, \xi_{k+2})$,
$\tau=(\tau_1,\cdots,\tau_{k+2})$ for short.

Third,  by symmetry we may assume that
\begin{equation*}
\aligned \big| \xi_{1}\big|\ge |\xi_2|\ge \cdots \geq
\big| \xi_{{k+2}} \big|.
\endaligned
\end{equation*}
Also, by dyadic decomposition, we may write
$$
|\xi_j|\sim N_j, \quad \mbox{ for } j=1,\cdots, k+2.
$$
According to the reduction in Remark \ref{rem:reduction}, we further assume $|N_1|\sim |N_2|\gtrsim N$.
The last reduction is that after replacing $\widehat{v_\lambda}(\tau,\xi)$ by $|\widehat{v_\lambda}(\tau,\xi)|$ if necessary, we assume that $\widehat{v_\lambda}(\tau,\xi)$ is positive.

Now we divide it into four
regions:
\begin{equation*}\aligned
  A_1 =& \{( \xi,\tau)\in (\Gamma_{k+2}\backslash\Omega)\times \Gamma_{k+2}:
  | \xi_2|\gtrsim N\gg | \xi_3|\}; \\
  A_2 =& \{( \xi,\tau)\in (\Gamma_{k+2}\backslash\Omega)\times \Gamma_{k+2}:
  | \xi_3|\gtrsim N\gg | \xi_4|\}; \\
  A_3 =& \{( \xi,\tau)\in (\Gamma_{k+2}\backslash\Omega)\times \Gamma_{k+2}:
  | \xi_4|\gtrsim N\gg | \xi_5|\}; \\
  A_4 =& \{( \xi,\tau)\in (\Gamma_{k+2}\backslash\Omega)\times \Gamma_{k+2}:
  | \xi_5|\gtrsim N\}.
\endaligned\end{equation*}

\noindent\textbf{Estimate in $A_1$.\quad} By Lemma \ref{es:Mk+2} (2), we have
$$
\big|M_{k+2}\big|\lesssim |\xi_3||\xi_4||\xi_5|.
$$
Therefore, by Lemma \ref{lem:XE6-12}  and Corollary \ref{cor:2.2}, we have
\begin{align*}
LHS \mbox{ of } (\ref{3.11})
\lesssim&
\iint_{A_1} |\xi_3||\xi_4||\xi_5|
\widehat{v_\lambda}(\tau_1,\xi_1)\cdots\widehat{v_\lambda}(\tau_{k+2},\xi_{k+2})\\
\lesssim&
N^{-2s+}\iint_{A_1} |\xi_1|^s|\xi_2|^{s-}|\xi_3||\xi_4|^{1-}|\xi_5|^{1-}
\widehat{v_\lambda}(\tau_1,\xi_1)\cdots\widehat{v_\lambda}(\tau_{k+2},\xi_{k+2})\\
\lesssim&
N^{-2+}\iint
\big(|\nabla|P_{N_1}Iv_\lambda\big)(t,x)\big(|\nabla|^{1-}P_{N_2}Iv_\lambda\big)(t,x)\big(|\nabla|P_{N_3}Iv_\lambda\big)(t,x)\\
&\qquad\cdot\big(|\nabla|^{1-}P_{\ll N }Iv_\lambda\big)^{2}(t,x)\big(P_{\ll N }v_\lambda\big)^{k-3}(t,x)\,dxdt \\
\lesssim&
N^{-2+}
\big\|\eta(t)^2I_{N^2}\big(|\nabla|P_{N_1}Iv_\lambda, |\nabla|P_{N_3}Iv_\lambda\big)\big\|_{L^2_{xt}}\big\||\nabla|^{1-}Iv_\lambda\big\|_{L^6_{xt}}^3\big\|v_\lambda\big\|_{L^\infty_{xt}}^{k-3} \\
\lesssim& \;
N^{-2+}\lambda^{-\frac12}.
\end{align*}

\noindent\textbf{Estimate in $A_2$.\quad} By the definition of $\Omega_1$, $A_2=\emptyset$.

\noindent\textbf{Estimate in $A_3$.\quad} We split it into two parts again, and define
\begin{equation*}\aligned
  A_{31} =& \{( \xi,\tau)\in A_3:
  | \xi_1|\sim | \xi_2|\sim| \xi_3|\sim| \xi_4|\}; \\
 A_{32} =& \{( \xi,\tau)\in A_3:
  | \xi_1|\sim | \xi_2|\gg| \xi_3|\sim| \xi_4|\}; \\
\endaligned\end{equation*}

\noindent{Estimate in $A_{31}$.\quad}
By Lemma \ref{es:Mk+2} (3),  we have
$$
|\chi_{A_{31}}\,{M}_{k+2}|\lesssim m( \xi_1 )^2| \xi_1 |^{2}|\xi_5 |\sim m( \xi_1 )m( \xi_2)| \xi_1 ||\xi_2|^{1-}|\xi_3|^{0+}|\xi_5|.
$$
Therefore, by Lemma \ref{lem:XE6-12} and Corollary \ref{cor:2.2},
\begin{equation*}\aligned
LHS \mbox{ of } (\ref{3.11})
\lesssim&
\iint_{A_3} m( \xi_1 )m( \xi_2)| \xi_1 ||\xi_2|^{1-}|\xi_3|^{0+}|\xi_5|
\widehat{v_\lambda}(\tau_1,\xi_1)\cdots\widehat{v_\lambda}(\tau_{k+2},\xi_{k+2})\\
\lesssim&
\iint
\big(|\nabla|P_{N_1}Iv_\lambda\big)(t,x)\big(|\nabla|^{1-}P_{N_2}Iv_\lambda\big)(t,x)\big(|\nabla|^{0+}P_{N_3}Iv_\lambda\big)(t,x)\\
&\qquad\qquad\cdot
 \big(P_{N_4}v_\lambda\big)(t,x)\big(|\nabla|P_{N_5}Iv_\lambda\big)(t,x)\big(P_{\ll N }v_\lambda\big)^{k-3}(t,x)\,dxdt \\
\lesssim&
\big\|\eta(t)^2I_{N^2}\big(|\nabla|P_{N_1}Iv_\lambda, |\nabla|P_{N_5}Iv_\lambda\big)\big\|_{L^2_{xt}}\big\||\nabla|^{1-}P_{N_2}Iv_\lambda\big\|_{L^{6}_{xt}}\\
&\qquad\qquad\cdot
\big\||\nabla|^{0+}P_{N_3}v_\lambda\big\|_{L^{6}_{xt}}\big\|P_{N_4}v_\lambda\big\|_{L^{6}_{xt}}\big\|P_{\ll N}v_\lambda\big\|_{L^\infty_{xt}}^{k-3} \\
\lesssim&
N^{-2+}
\big\|\eta(t)^2I_{N^2}\big(|\nabla|P_{N_1}Iv_\lambda, |\nabla|P_{N_5}Iv_\lambda\big)\big\|_{L^2_{xt}}
\big\||\nabla|^{1-}Iv_\lambda\big\|_{L^{6}_{xt}}^3\big\|v_\lambda\big\|_{L^\infty_{xt}}^{k-3} \\
\lesssim& \;
N^{-2+}\lambda^{-\frac12}.
\endaligned\end{equation*}

\noindent{Estimate in $A_{32}$.\quad}
Note that both the estimates in Lemma \ref{es:Mk+2} (1) and (3) hold in $A_{32}$, so for any $\epsilon>0$,
\begin{align*}
|\chi_{A_{32}}\>{M}_{k+2}|\lesssim &
\big[m( \xi_1 )^2| \xi_1 |^{2}|\xi_5 |\big]^{1-\epsilon}\big[m( \xi_1 )^2| \xi_1 ||\xi_3|^{2}\big]^\epsilon\\
= &
m( \xi_1 )^2| \xi_1 |^{2-\epsilon}|\xi_3|^{2\epsilon}|\xi_5|^{1-\epsilon}\\
\lesssim &
m( \xi_1 )^2| \xi_1 ||\xi_2|^{1-\epsilon}|\xi_3|^{\epsilon}|\xi_4|^{\epsilon}\langle \xi_5\rangle.
\end{align*}
Therefore, by Lemma \ref{lem:XE6-12} and Corollary \ref{cor:2.2},
\begin{equation*}\aligned
LHS \mbox{ of } (\ref{3.11})
\lesssim&
\iint_{A_3} m( \xi_1 )^2| \xi_1 ||\xi_2|^{1-\epsilon}|\xi_3|^{\epsilon}|\xi_4|^{\epsilon}\langle \xi_5\rangle
\widehat{v_\lambda}(\tau_1,\xi_1)\cdots\widehat{v_\lambda}(\tau_{k+2},\xi_{k+2})\\
\lesssim&
\iint
\big(|\nabla|P_{N_1}Iv_\lambda\big)(t,x)\big(|\nabla|^{1-\epsilon}P_{N_2}Iv_\lambda\big)(t,x)\big(|\nabla|^\epsilon P_{N_3}v_\lambda\big)^2(t,x)\\
&\qquad\qquad\cdot
 \big(\langle\nabla\rangle P_{N_5}Iv_\lambda\big)(t,x)\big(P_{\ll N }v_\lambda\big)^{k-3}(t,x)\,dxdt \\
\lesssim&
\big\|\eta(t)^2I_{N^2}\big(|\nabla|P_{N_1}Iv_\lambda, \langle\nabla\rangle P_{N_5}Iv_\lambda\big)\big\|_{L^2_{xt}}\big\||\nabla|^{1-}P_{N_2}Iv_\lambda\big\|_{L^{6}_{xt}}\\
&\qquad\qquad\cdot
\big\||\nabla|^{0+}P_{N_3}v_\lambda\big\|_{L^{6}_{xt}}^2\big\|P_{\ll N}v_\lambda\big\|_{L^\infty_{xt}}^{k-3} \\
\lesssim&
N^{-2+}
\big\|\eta(t)^2I_{N^2}\big(|\nabla|P_{N_1}Iv_\lambda, \langle\nabla\rangle P_{N_5}Iv_\lambda\big)\big\|_{L^2_{xt}}
\big\||\nabla|^{1-}Iv_\lambda\big\|_{L^{6}_{xt}}^3\big\|v_\lambda\big\|_{L^\infty_{xt}}^{k-3} \\
\lesssim& \;
N^{-2+}\lambda^{-\frac12}.
\endaligned\end{equation*}

\noindent\textbf{Estimate in $A_4$.\quad}
Moreover, we split $A_4$ into two subregions:
\begin{equation*}\aligned
  A_{41} =& \{( \xi,\tau)\in A_5:
  | \xi_4|\gg | \xi_5|\}; \\
  A_{42} =& \{( \xi,\tau) \in A_5:
  | \xi_4|\sim | \xi_5|\}.
\endaligned\end{equation*}
The estimate in $A_{41}$ can be treated as the estimate in $A_{3}$, since they have the same bound on $M_{k+2}$. So we omit the details. Now we consider the estimate in $A_{42}$. In this part, by Lemma \ref{es:Mk+2} (1) and the relationship $|\xi_3|\sim |\xi_4|\sim |\xi_5|$, we have
\begin{align*}
|M_{k+2}|\lesssim m(\xi_1)^2|\xi_1||\xi_3|^2
\lesssim &
m(\xi_1)|\xi_1|\cdot m(\xi_2)|\xi_2|^{\frac12+3\epsilon}|\xi_3|^{\frac12-\epsilon}|\xi_4|^{\frac12-\epsilon}|\xi_5|^{\frac12-\epsilon}\\
\lesssim &
N^{-\frac32-}m(\xi_1)|\xi_1|\cdot m(\xi_2)|\xi_2|^{\frac12+}\cdot m(\xi_3)|\xi_3|^{1-}\cdot m(\xi_4)|\xi_4|^{1-}\cdot m(\xi_5)|\xi_5|\\
\lesssim &
N^{-2+}m(\xi_1)|\xi_1|\cdot m(\xi_2)|\xi_2|^{1-}\cdot m(\xi_3)|\xi_3|^{1-}\cdot m(\xi_4)|\xi_4|^{1-}\cdot m(\xi_5)|\xi_5|.
\end{align*}
Further, we claim that
\begin{equation}\label{7.25}
| \xi_1|-| \xi_5|\gtrsim | \xi_1|\gtrsim N.
\end{equation}
Indeed, if $| \xi_j|=| \xi_1|+o(| \xi_1|)$, for all $j=1,\cdots, 5$, then there exist $\mu_j\in \{-1,1\}$ such that
$$
\xi_j= \mu_j \xi_1+o(| \xi_1|).
$$
Therefore,
$$
|\xi_6|=\big|\xi_1+\cdots+ \xi_5\big|=|\mu_1+\cdots+\mu_5||\xi_1|+o(|\xi_1|).
$$
Note that $|\mu_1+\cdots+\mu_5|\ge 1$, we have $|\xi_6|\sim |\xi_1|$, but this is not the case in $A_4$. So we have $| \xi_1|-| \xi_5|\gtrsim | \xi_1|$ and thus
$$
\big|\xi_1^2-\xi_5^2\big|\gtrsim N^2.
$$
Therefore,  we have
\begin{equation*}\aligned
LHS \mbox{ of } (\ref{3.11})
\lesssim&
N^{-2+}\iint_{A_4} m(\xi_1)|\xi_1|\cdot m(\xi_2)|\xi_2|^{1-}\cdot m(\xi_3)|\xi_3|^{1-}\cdot m(\xi_4)|\xi_4|^{1-}\cdot m(\xi_5)|\xi_5|\\
&\qquad \qquad\cdot
\widehat{v_\lambda}(\tau_1,\xi_1)\cdots\widehat{v_\lambda}(\tau_{k+2},\xi_{k+2})\\
\lesssim&
N^{-2+}\iint \eta^2(t)
I_{N^2}\big(|\nabla|Iv_\lambda, |\nabla|Iv_\lambda\big)(t,x)\\
&\qquad \qquad\cdot
\big(|\nabla|^{1-}P_{\gtrsim N}Iv_\lambda\big)^3(t,x)
\big(P_{\ll N }v_\lambda\big)^{k-3}(t,x)\,dxdt  \\
\lesssim&
N^{-2+}\big\|\eta(t)^2I_{N^2}\big(|\nabla|Iv_\lambda, |\nabla|Iv_\lambda\big)\big\|_{L^2_{xt}}
\big\||\nabla|^{1-}Iv_\lambda\big\|_{L^{6}_{xt}}^3\big\|v_\lambda\big\|_{L^\infty_{xt}}^{k-3} \\
\lesssim& \;
N^{-2+}\lambda^{-\frac12}.
\endaligned\end{equation*}

\noindent\textbf{Estimate in $A_5$ (if $k=4$).\quad} It also can be treated as the estimate in $A_4$, and thus we obtain the same conclusion as what in $A_4$.
\end{proof}

\begin{lem} \label{2p+2-linear}
Let $s\geq\frac{1}{2}$, and $ \|Iu\|_{X^1([0,\delta])}\lesssim  1$, then
\begin{equation}
 \left|\int_0^t\!\!\int_{\Gamma_{2k+2}}\!\! e^{i\alpha_{2k+2}s}\overline{M_{2k+2}}\>
\widehat{f_\lambda}(s,\xi_1)\cdots\widehat{f_\lambda}(s,\xi_{2k+2})\right|
 \lesssim
 K'\>
 \|Iu_\lambda\|_{Y^1}^{2k+2}.
 \label{Lambda6}
\end{equation}
where $K'=N^{-3+}+N^{-2+}\lambda^{-1}$.
\end{lem}
\begin{proof}
By the reductions at the beginning of the proof of
Proposition \ref{p+2-linear}, it suffices to show
\begin{equation}\label{3.11}
 \left|\iint_{\Gamma_{2k+2}\times \Gamma_{2k+2}} \overline {M_{2k+2}}
\widehat{v_\lambda}(\tau_1,\xi_1)\cdots\widehat{v_\lambda}(\tau_{2k+2},\xi_{2k+2})\right|
\lesssim K'\>\|Iv_\lambda\|_{Y^1}^{2k+2},
\end{equation}
where the hyperplane $\Gamma_{2k+2}^2
=\{( \xi_1,\cdots, \xi_{2k+2},\tau_1,\cdots,\tau_{2k+2}): \xi_1+\cdots+ \xi_{2k+2}=0,
\tau_1+\cdots+\tau_{2k+2}=0\}$.
Also, we may assume that
\begin{equation*}
\aligned | \xi_{1}|\ge |\xi_2|\ge \cdots \geq
| \xi_{{2k+2}} |,\quad  |\xi_j|\sim N_j, \mbox{ for } j=1,\cdots, 2k+2,
\endaligned
\end{equation*}
and  $\widehat{v_\lambda}(\tau,\xi)$ is positive.

Now we consider the following three subregions separately:
\begin{equation*}\aligned
  B_1 =& \{( \xi_1,\cdots, \xi_{2k+2},\tau_1,\cdots,\tau_{2k+2})\in \Gamma_{2k+2}^2:
  | \xi_1|\sim| \xi_2|\gtrsim N\gg | \xi_3|\}; \\
  B_2 =& \{( \xi_1,\cdots, \xi_{2k+2},\tau_1,\cdots,\tau_{2k+2})\in \Gamma_{2k+2}^2:
  | \xi_1|\sim| \xi_2|\geq | \xi_3|\gtrsim N \gg |\xi_4|\};\\
  B_3 =& \{( \xi_1,\cdots, \xi_{2k+2},\tau_1,\cdots,\tau_{2k+2})\in \Gamma_{2k+2}^2:
  | \xi_4|\gtrsim N\}.
\endaligned\end{equation*}

\noindent\textbf{Estimate in $B_1$.\quad}
By Lemma \ref{es:M2k+2} (2), we have $|\overline{M_{2k+2}}|\lesssim|\xi_3|$. Then, by
(\ref{CEIN}), we have
\begin{equation*}\aligned
\mbox{LHS of (\ref{3.11})}\lesssim &\; \int_{B_1} |\xi_3|
\widehat{v_\lambda}(\tau_1,\xi_1)\cdots\widehat{v_\lambda}(\tau_{2k+2},\xi_{2k+2})\\
\lesssim &\; N^{-2}\int_{B_1} m(\xi_1)|\xi_1|\cdot m(\xi_2)|\xi_2|\cdot m(\xi_3)|\xi_3|\cdot m(\xi_4)\langle\xi_4\rangle
\widehat{v_\lambda}(\tau_1,\xi_1)\cdots\widehat{v_\lambda}(\tau_{2k+2},\xi_{2k+2})\\
\lesssim&\; N^{-2}\big\|\eta(t)^2I_{N^2}\big(|\nabla|P_{N_1}Iv_\lambda, |\nabla| P_{N_3}Iv_\lambda\big)\big\|_{L^2_{xt}}\\
&\quad\cdot
\big\|\eta(t)^2I_{N^2}\big(|\nabla|P_{N_2}Iv_\lambda, \langle\nabla\rangle P_{N_4}Iv_\lambda\big)\big\|_{L^2_{xt}}\big\|v_\lambda\big\|_{L^\infty_{xt}}^{2k-2}\\
\lesssim&\; N^{-2+}\lambda^{-1}. \endaligned\end{equation*}

\noindent\textbf{Estimate in $B_2$.\quad} Similar to the proof of \eqref{7.25}, we have
\begin{equation}\label{7.26}
| \xi_1|-| \xi_3|\gtrsim | \xi_1|\gtrsim N.
\end{equation}
Moreover, from \eqref{M2k+2-1}, we have $|M_{2k+2}|\lesssim |\xi_1|$. Then
\begin{align*}
|M_{2k+2}|\lesssim  |\xi_1|
\lesssim 
N^{-2}m(\xi_1)|\xi_1|\cdot m(\xi_2)|\xi_2|\cdot m(\xi_3)|\xi_3|\cdot m(\xi_4)\langle\xi_4\rangle.
\end{align*}
Therefore, we have the same estimate as what in $B_1$, and get also
\begin{equation*}\aligned
\mbox{LHS of (\ref{3.11})}
\lesssim N^{-2+}\lambda^{-1}.
\endaligned\end{equation*}

\noindent\textbf{Estimate in $B_3$.\quad} In this part,
\begin{align*}
|M_{2k+2}|\lesssim  |\xi_1|
\lesssim 
N^{-3+s}N_4^{-s}m(\xi_1)|\xi_1|\cdot m(\xi_2)|\xi_2|\cdot m(\xi_3)|\xi_3|\cdot m(\xi_4)|\xi_4|.
\end{align*}
Therefore,  we have
\begin{equation*}\aligned
\mbox{LHS of (\ref{3.11})}
\lesssim &\; N^{-3+s}N_4^{-s}\int_{B_3} m(\xi_1)|\xi_1|\cdot m(\xi_2)|\xi_2|\cdot m(\xi_3)|\xi_3|\cdot m(\xi_4)|\xi_4|\\
&\qquad \qquad\cdot
\widehat{v_\lambda}(\tau_1,\xi_1)\cdots\widehat{v_\lambda}(\tau_{2k+2},\xi_{2k+2})\\
\lesssim&\; N^{-3+s}N_4^{-s}\big\||\nabla|P_{N_1}Iv_\lambda\big\|_{L^4_{xt}}
\cdots
\big\||\nabla|P_{N_4}Iv_\lambda\big\|_{L^4_{xt}}
\|v_\lambda\|_{L^\infty_{xt}}^{2k-2}\\
\lesssim&\; N^{-3+s}N_4^{-s+}\\
\lesssim&\;  N^{-3+}.
\endaligned\end{equation*}
This gives the proof of the lemma.
\end{proof}

Since $K'\le K$, combining with the results on Lemma \ref{lem:fixed-time}--Lemma \ref{2p+2-linear}, we prove Proposition \ref{thm:main-2}.

\vspace{0.3cm}
\section{Proposition \ref{thm:main-2} implies Theorem \ref{thm:main}}

Suppose that
\begin{align}\label{2.288}
\sup\limits_{t\in [0,(j-1)\delta]}E(Iu_\lambda(t))\le 2E(I\phi_\lambda)\quad \mbox{for some } j\in \mathbb N.
\end{align}
Then by local theory in Lemma \ref{lem:modified-local}, we have
\begin{equation*}
    \|Iu_\lambda\|_{X^1([(i-1)\delta,i\delta])}
    \lesssim
    1, \quad \mbox{for any } 1\le  i \le j.
\end{equation*}
So by \eqref{Almost conserved} in Proposition \ref{thm:main-2} and a simple iteration, we have
\begin{align}\label{2.16}
\left|E^2_I(t)-E^2_I(0)\right|&\le j K\notag\\
&\le  C_0:=\frac12E(I\phi_\lambda)
 \end{align}
for any $ t\le j\delta, j\le C_0 K^{-1}$. By \eqref{fixed-time bound}, we have
$$
\big|E(Iu(t))-E_I^2(t)\big|\lesssim N^{-2+}\|Iu(t)\|^{k+2}_{H^1_x}\lesssim N^{-2+}E(Iu(t))^\frac{k+2}2.
$$
This combining with \eqref{2.16},  gives us that for any  $ t\le j\delta$, $j$ satisfying \eqref{2.288} and $\le C_0 K^{-1}$,
\begin{align*}
E(Iu(t))&\le E_I^2(t)+CN^{-2+}E(Iu(t))^\frac{k+2}2\\
&\le
E_I^2(0)+\frac12E(I\phi_\lambda)+CN^{-2+}E(Iu(t))^\frac{k+2}2\\
&\le
\frac32E(I\phi_\lambda)+CN^{-2+}\big(E(Iu(t))^\frac{k+2}2+E(I\phi_\lambda)^\frac{k+2}2\big).
\end{align*}
So by continuity argument, we have for any $ t\le j\delta$,
$$
E(Iu(t))\le 2E(I\phi_\lambda).
$$
This extends \eqref{2.288} to $[0,j\delta]$. Thus by finite induction, we obtain that
$$
\sup\limits_{t\in [0,C_0\delta K^{-1}]}E(Iu_\lambda(t))\le 2E(I\phi_\lambda).
$$
This proves that $u_\lambda$ exists on $[0,C_0\delta K^{-1}]$, which implies that $u$ exists on $[0,C_0\delta \lambda^{-3} K^{-1}]$. Suppose that
$$
\delta\lambda^{-3} K^{-1}\ge N^{0+},
$$
then $u$ exists for arbitrary time by choosing large $N$.

Since
$$
\delta\sim \lambda^{0+},\quad \lambda\sim N^{\frac{1-s}{\frac{2}{k}+s-\frac{1}{2}}}, \quad
K= N^{-3+}+N^{-2+}\lambda^{-\frac12},
$$
we have $\delta\lambda^{-3} K^{-1}\ge N^{0+}$ as long as
\begin{eqnarray*}
2>\frac52\cdot\frac{1-s}{\frac{2}{k}+s-\frac{1}{2}};\quad
3>\frac{3(1-s)}{\frac{2}{k}+s-\frac{1}{2}}.
\end{eqnarray*}
Particularly, when $k=3$, it holds for any $s\ge \frac12$; when $k=4$,
it holds for any $s> \frac59$. This completes the proof of
Theorem 1.1.

\end{document}